\documentclass[11pt,letterpaper,reqno]{amsart}
\usepackage[left=30mm,top=40mm,right=30mm,bottom=35mm]{geometry}
\usepackage{epsfig}
\usepackage{centernot}
\usepackage{amssymb}
\usepackage{mathrsfs}
\usepackage{amsthm}
\usepackage{subcaption}
\usepackage{multicol}
\usepackage{times}
\usepackage{color}
\usepackage{enumerate}
\usepackage{subcaption} 
\usepackage{graphicx}
\usepackage{nicefrac}
\usepackage{xcolor}
\usepackage[all]{xy}
\usepackage{comment}
\usepackage{hyperref}
\hypersetup{
 colorlinks=true, 
 linkcolor=blue, 
 citecolor=blue, 
 filecolor=blue, 
 urlcolor=blue 
}
\usepackage[nameinlink,capitalize,noabbrev]{cleveref}

\newtheorem{theorem}{Theorem}[section]
\newtheorem*{theorem*}{Theorem}
\newtheorem{corollary}[theorem]{Corollary}
\newtheorem{lemma}[theorem]{Lemma}
\newtheorem{proposition}[theorem]{Proposition}

\newtheorem{remark}[theorem]{Remark}

\newtheorem*{claim*}{Claim}

{\theoremstyle{definition}
 \newtheorem{definition}[theorem]{Definition}

 \newtheorem{example}[theorem]{Example}

}

\newcommand{\N}{\mathbb N}
\newcommand{\Z}{\mathbb Z}
\newcommand{\Q}{\mathbb Q}
\newcommand{\R}{\mathbb R}

\newcommand{\K}{\mathbf{k}}
\newcommand{\sv}{\sigma^{\vee}}
\newcommand{\gsv}{G(\sv)}
\newcommand{\xs}{X(\sigma)}

\newcommand{\dez}{\operatorname{det}_0}
\newcommand{\jp}{\mathcal{J}_p(\sigma)}
\newcommand{\jz}{\mathcal{J}_0(\sigma)}
\newcommand{\A}{\mathcal{A}}

\newcommand{\npc}{\mathcal{N}_0(\omega_{P})}
\newcommand{\npp}{\mathcal{N}_p(\omega_{P})}

\newcommand{\sip}{\omega_{P}}

\newcommand{\sd}{\sigma^{\vee}}

\newcommand{\GM}[0]{\ensuremath{\mathbb{G}_{\mathrm{m}}}}
\newcommand{\skel}{\operatorname{skel}}

\newcommand{\con}[1]{\operatorname{Cone}{(#1)}}

\begin{document}

\title{Characteristic-free normalized Nash blowup of toric varieties}

\author[{F. Castillo}]{Federico Castillo}
\address{Facultad de Matemáticas, Pontificia Universidad Católica de Chile, Santiago, Chile}  
\email{federico.castillo@uc.cl}

 \author[D. Duarte]{Daniel Duarte}
\address{Centro de Ciencias Matematicas, UNAM, Morelia, M\'exico}  \email{adduarte@matmor.unam.mx}

\author[M. Leyton-Alvarez]{Maximiliano Leyton-\'Alvarez} %
\address{Instituto de Matem\'aticas, Universidad de Talca, Talca, Chile} %
\email{mleyton@utalca.cl}

\author[A. Liendo]{Alvaro Liendo} %
\address{Instituto de Matem\'aticas, Universidad de Talca, Talca, Chile} %
\email{aliendo@utalca.cl}

\date{\today}

\thanks{
 \mbox{\hspace{11pt}}{\it Key words}:  Resolution of singularities, Nash blowup, Toric varieties.\\
 \mbox{\hspace{11pt}} The first author was partially supported by Fondecyt project 1221133. The second author was partially supported by CONAHCYT project CF-2023-G-33 and PAPIIT grant IN117523. The third author was partially supported by Fondecyt project 1221535. Finally, the fourth author was partially supported by Fondecyt project 1240101.}

\begin{abstract}
We introduce conditions on cones of normal toric varieties under which the polyhedron defining the normalized Nash blowup does not depend on the characteristic of the base field. As a consequence, we deduce several results on the resolution of singularities properties of normalized Nash blowups. 
In particular, we recover all known results of the families that can be resolved via normalized Nash blowups in positive characteristic.
We also provide new families of toric varieties whose normalized Nash blowup is non-singular in arbitrary characteristic.
\end{abstract}

\maketitle

 \tableofcontents

\section*{Introduction}

Let $X \subseteq \mathbb{A}^n_{\K}$ be an equidimensional affine algebraic variety of dimension $d$, where $\K$ is an algebraically closed field of arbitrary characteristic. Consider the Gauss map:
$$
\mathcal{G} : X \setminus \operatorname{Sing}(X) \to \operatorname{Grass}(d, n)\quad\mbox{given by}\quad x \mapsto T_x X\,,$$
where $\operatorname{Grass}(d, n)$ is the Grassmannian of $d$-dimensional vector spaces in $\K^n$, and $T_x X$ is the tangent space to $X$ at $x$. Denote by $X^*$ the Zariski closure of the graph of $\mathcal{G}$. Let $\nu\colon X^*\to X$ be the restriction to $X^*$ of the projection from $X \times \operatorname{Grass}(d, n)$ to $X$. The map $\nu$ is a birational morphism that is an isomorphism over $X\setminus\operatorname{Sing}(X)$ called the \emph{Nash blowup} of $X$, and the composition $\eta\circ\nu$, where $\eta$ is the normalization map is called the \emph{normalized Nash blowup} of $X$ \cite{Nob75,GS1}.

The resolution properties of Nash blowups and their normalized versions 
have been studied extensively by many authors 
\cite{Nob75,Reb,GS1,GS2,Hi83,Sp90,GS3,EGZ,Ataetal,GrMi12,GoTe14,DJNB,DDR}.

Recently, the authors of this paper showed that iterating Nash blowups fails to provide a resolution of singularities in dimensions four and higher \cite{CDLAL}. This was shown by constructing explicit counterexamples of normal affine varieties whose Nash blowup contains an affine chart isomorphic to the variety itself. We also remark that the counterexamples in \cite{CDLAL} were obtained using the software implementation developed for this project in SAGE \cite{sagemath}. The main open questions concerning the resolution properties of Nash blowups remain in dimensions 2 and 3; see \cite[Introduction]{CDLAL} for further details.

The case of toric varieties has received special attention. In fact, the known counterexamples to Nash's question are toric varieties. It is then crucial to first settle the toric case in dimensions two and three before exploring other families. This paper aims at providing new tools for the study of normalized Nash blowups of toric varieties, with a strong emphasis on the role of the characteristic of the base field.

The content of this paper was inspired by recent results on Nash blowups of normal toric surfaces and 2-generic determinantal varieties (which are also normal toric varieties) over fields of positive characteristic. In both cases, it was proved that the normalized Nash blowup resolves their singularities by showing that the combinatorics of the objects defining the Nash blowup are independent of the characteristic \cite{DJNB,DDR}. The results were thus deduced from the characteristic zero case \cite{GS1,EGZ}. Nevertheless, it is known that the combinatorics of Nash blowups do depend on the characteristic in general \cite{DJNB}. It is then of interest to find out under which conditions the Nash blowup is free of characteristic.

In this paper we introduce two different but related conditions on a cone under which the combinatorial objects defining the normalized Nash blowup of the corresponding toric variety are independent of the characteristic. The first condition is called $G$-stability (see \cref{def:DCMGS}), the second one is related to lattice polytopes and the cones they generate (see Definitions \ref{def:G-flat} and \ref{def:smooth}). Both notions were partly inspired by the two-dimensional case, where minimal generating sets are determined by convex conditions \cite{oda1983convex}.

To state our main theorems, first recall that the normalized Nash blowup of a toric variety $X(\sigma)$ corresponds to the blowup of the logarithmic Jacobian ideal $\mathcal{J}_p(\sigma)$, where $p\geq0$ is the characteristic of the base field \cite{GS1,DJNB}. Being a monomial ideal, its normalized blowup is determined by some Newton polyhedron, which we denote as $\mathcal{N}_p(\sigma)$. Our two main theorems study this Newton polyhedron.

\begin{theorem*}(see \cref{th:Nash polyh} and \cref{cor:char_free})
Let $\sigma$ be a full-dimensional strongly convex cone. Suppose that $\sigma^\vee$ is $G$-stable or it is defined by a smooth $G$-flat lattice polytope. Then $\mathcal{N}_0(\sigma)=\mathcal{N}_p(\sigma)$ for all $p>0$ prime.    
\end{theorem*}

As a consequence, any property of the normalized Nash blowup of the toric variety $X(\sigma)$ that corresponds to a combinatorial property of the Newton polyhedron, holds over a characteristic zero field if and only if it holds over a prime characteristic field. As particular cases we recover one of the main theorems of \cite{DJNB} and a normalized version of the main theorem of \cite{DDR}.

Finally, as a consequence of our work on lattice polytopes, we can construct a family of normal toric varieties whose normalized Nash blowup is non-singular in arbitrary characteristic.

\begin{theorem*}(see \cref{thm:altura1})
Let $ P\subset M_{\R} $ be a smooth and $G$-flat polytope. Let $X$ be the normal toric variety defined by the cone generated by $\{(v,1)\mid v\in P\}\subset M_{\R}\times\R$. Then the normalized Nash blowup of $X$ is non-singular over fields of arbitrary characteristic.
\end{theorem*}

To prove this theorem, we introduce the notion of barycentric hull (see \cref{def:barycentric-hull}). We show that the barycentric hull of a smooth polytope $P$ has a shape similar to $P$ (see \cref{thm:baryhull}). This result may be of independent interest.

In the particular case where $P$ is a polygon, $G$-flatness is a consequence of smoothness. Hence the three-dimensional family of toric varieties defined by smooth polygons as in the theorem is resolved by a normalized Nash blowup. This provides evidence towards a positive answer on the question of the resolution properties of the normalized Nash blowup in dimension three.

\medskip

The paper is divided as follows. We first recall the basic facts from toric geometry that we need. Section 2 is devoted to the study of $G$-stability and its implications on the normalized Nash blowup. Section 3 studies Nash blowups of toric varieties defined by cones associated to lattice polytopes. Both sections contain a great deal of examples illustrating the notions we introduce. Many of these examples were obtained using the software SAGE \cite{sagemath}.

\subsection*{Acknowledgments}

This collaboration began during the third named author's visit to the Centro de Ciencias Matemáticas, UNAM Campus Morelia, in September 2023. The first and fourth authors joined the project at the \href{https://sites.google.com/view/agrega0/home}{AGREGA} workshop, held at the Universidad de Talca in January 2024, where the second named author presented the Nash blowup conjecture. We extend our gratitude to these institutions for their support and hospitality. We also thank Takehiko Yasuda for explaining the content of \cref{rem:moderate} to us.

\section{Preliminaries on toric varieties}\label{sec:toric}

A toric variety $X$ is a normal variety endowed with a faithful regular action $T\times X\rightarrow X$ of the algebraic torus $T=\GM^d$ having an open orbit, where $\GM$ is the multiplicative group of the base field $\K$. Toric varieties admit a combinatorial description that we recall now, for details, see \cite{oda1983convex,fulton1993introduction,cox2011toric,GoTe14}. 

Denote as $M$ (resp. $N$) the set of characters (resp. 1-parameter subgroups) of $T$. These are dual free abelian groups of rank $\dim T$. There is a natural duality pairing $M\times N\rightarrow \Z$, $(u,p)\mapsto \langle u,p\rangle$. We also let $M_\R=M\otimes_\Z \R$ and $N_\R=N\otimes_\Z \R$ be the corresponding real vector spaces and we extend the duality to them.

A fan $\Sigma$ in $ N_\R$ is a finite collection of strictly convex rational polyhedral cones such that every face of $\sigma\in\Sigma$ is contained in $\Sigma$ and for all $\sigma,\sigma'\in\Sigma$ the intersection $\sigma\cap\sigma'$ is a face of each cone $\sigma,\sigma'$. We refer to strictly convex rational polyhedral cones simply as cones. A face $\tau$ of a cone $\sigma$ is denoted $\tau\prec\sigma$. We regard a cone as a fan by taking the set of all faces of $\sigma$. Moreover, if the dimension of a cone $\sigma$ is equal to the rank of $N$, we say that $\sigma$ is full-dimensional.
 
 Let $\Sigma$ be a fan in $N_\R$. A toric variety $X(\Sigma)$ is obtained from $\Sigma$ in the following way: for every cone $\sigma\in\Sigma$, we define an affine toric variety $X(\sigma)=\operatorname{Spec}\K[\sigma^{\vee}\cap M]$, where $\sigma^{\vee}$ is the dual cone of $\sigma$ in $M_{\R}$ and $\K[\sigma^{\vee}\cap M]$ is the semigroup algebra 
\[
\K[\sigma^{\vee}\cap M]=\bigoplus_{u\in\sigma^{\vee}\cap M}\K\cdot\chi^{u},\quad\mbox{with}\quad\chi^{0}=1,\mbox{ and }\chi^{u}\cdot\chi^{u'}=\chi^{u+u'},\ \forall u,u'\in\sigma^{\vee}\cap
M\,.
\]
We define the toric variety $X(\Sigma)$ associated with the fan $\Sigma$ as the variety obtained by gluing the family $\{X(\sigma)\mid\sigma\in\Sigma\}$ along the open embeddings $X(\sigma)\hookleftarrow X(\sigma\cap\sigma')\hookrightarrow X(\sigma')$ for all $\sigma,\sigma'\in\Sigma$. 

The set of $i$-dimensional cones of $\Sigma$ is denoted by $\Sigma(i)$. The 1-dimensional cones $\Sigma(1)$ of $\Sigma$ are called rays. A ray of $\rho\in\Sigma(1)$ is determined by its unique primitive vector. Hence, by abuse of notation, we consider that $\Sigma(1)$ is the set of primitive vectors of the rays of $\Sigma$. The $i$-skeleton of $\Sigma$, denoted by $\skel_i(\Sigma)$ is the set of faces of dimension less or equal than $i$, i.e., $\skel_i(\Sigma)=\bigcup_{j=0}^i \Sigma(i)$. The dimension of a fan $\Sigma$ in $N_\R$ is the largest integer $n$ such that $\Sigma(n)\neq\emptyset$. We always have $n=\dim\Sigma\leq \operatorname{rank}N$.

The support  $|\Sigma|$ of a fan $\Sigma$ in $N_\R$ is the union of the cones in $\Sigma$, i.e., $|\Sigma|=\bigcup_{\sigma\in \Sigma}\sigma$.  A refinement of $\Sigma$ is fan $\Sigma'$  such that $|\Sigma|=|\Sigma'|$ and for all $\sigma'\in \Sigma'$ there exists $\sigma\in \Sigma$ with $\sigma'\subset \sigma$. We denote a refinement by  $\Sigma' \prec_\textrm{ref} \Sigma$.

Given a cone $\sigma \subset N_{\R}$, we let $\mathcal{S}(\sigma)$ be the semigroup $\mathcal{S}(\sigma)=\sigma\cap N$. Its minimal generating set is denoted as $G(\sigma)$.
This set is usually called the Hilbert basis of $\sigma$.
Given a fan $\Sigma\subset N_\R$ we denote by $G(\Sigma)$ the union of all $G(\sigma)$ for $\sigma\in \Sigma$. 

A cone $\sigma$ is called simplicial if $\sigma(1)$ is part of an $\R$-basis of $N_\R$ and regular if $\sigma(1)$ is part of a $\Z$-basis of $N$.
A cone $\sigma\subset N_\R$ is regular if and only if $\sigma$ is simplicial and $G(\sigma)=\sigma(1)$ \cite[Proposition~11.1.8]{cox2011toric}. A fan $\Sigma$ is called regular (resp. simplicial) if every cone $\sigma\in \Sigma$ is regular (resp. simplicial). The toric variety $X(\Sigma)$ is smooth if and only if $\Sigma$ is regular.

A refinement  $\Sigma'\prec_{\mathrm{ref}}\Sigma$ induces a proper birational morphism  $X(\Sigma')\to X(\Sigma)$. A $G$-desin\-gu\-lar\-iza\-tion of a toric variety $X(\Sigma)$ is a proper birational morphism $X(\Sigma')\rightarrow X(\Sigma)$, where $\Sigma'$ is a refinement of $\Sigma$ that is regular and, moreover, $\Sigma'(1)=G(\Sigma)$. In this case $\Sigma'$ is called a regular $G$-refinement of $\Sigma$.

\section{G-stability}\label{sec:G-stability}

This section is devoted to the study of the notion of $G$-stability. We provide several examples, prove basic properties as well as the first theorem stated in the introduction.

Let $\sigma\subset N_\R$ be a cone. We denote $\Gamma_{+}(\sigma)=\operatorname{Conv}(\mathcal{S}(\sigma)\setminus \{0\})$. In addition, we denote as $\Gamma(\sigma)$ the union of the compact faces of the polyhedron $\Gamma_{+}(\sigma)$. Given $A\subset N$, the cone generated by $A$ in $N_\R$ is denoted by $\sigma_A$ or as $\operatorname{Cone}(A)$.

\begin{definition}\label{def:DCMGS}
A cone $\sigma\subset N_{\R}$ is said to be $G$-stable if 
\begin{enumerate}
 \item[$(i)$] $G(\sigma)=\Gamma(\sigma)\cap N$.
 \item[$(ii)$] $G(\sigma_A)=G(\sigma)\cap \sigma_A$ for every $A\subset G(\sigma)$.
\end{enumerate}
A fan $\Sigma$ in $N_{\R}$ is $G$-stable if every cone $\sigma$ in $\Sigma$ is $G$-stable.
\end{definition}

We present several examples in order to get familiarity with this notion. Some of these examples were obtained using the software Sagemath \cite{sagemath}.

\begin{example}\label{ex:surf-one-segment}
Let $\sigma=\con{(1,0),(1,n)}\subset\R^2$ for some $n\in\Z_{\geq1}$. Then $\sigma$ is $G$-stable. Indeed, the minimal generating set of $\mathcal{S}(\sigma)$ is $G(\sigma)=\{(1,0),(1,1),\ldots,(1,n)\}\subset\Z^2$. Moreover, $\Gamma_+(\sigma)=\sigma\cap(\R_{\geq1}\times\R)$ and $\Gamma(\sigma)$ is the segment in $\R^2$ joining $(1,0)$ and $(1,n)$. In particular, \cref{def:DCMGS} $(i)$ is satisfied. Let us verify condition $(ii)$. Let $A=\{(1,i_1),\ldots,(1,i_k)\}\subset G(\sigma)$. If $k=1$, condition $(ii)$ is verified. Assume $i_1<\cdots<i_k$, where $k\geq2$. Then $\sigma_A=\con{(1,i_1),(1,i_k)}$ and $G(\sigma_A)=\{(1,i_1),(1,i_1+1),\ldots,(1,i_k-1),(1,i_k)\}=G(\sigma)\cap\sigma_A$.
\end{example}

As we will later, all two-dimensional cones are $G$-stable (see \cref{prop:G-stable-dim-two}). Actually, the notion of $G$-stability is partly inspired by the description of minimal generating sets in the two-dimensional case. Indeed, in that case it is well-known that \cref{def:DCMGS} $(i)$ holds \cite[Proposition 1.21]{oda1983convex}.

\begin{example}\label{ex:no-c1}
This example exhibits a cone that does not satisfy \cref{def:DCMGS}~$(i)$. Denote  $\sigma=\con{(1,0,0), (0,1,0), (1,1,2)} \subset N_{\R}$. The minimal generating set of the corresponding semigroup is $G(\sigma)=\{(1,0,0), (0,1,0), (1,1,2), (1,1,1)\}.$ The polyhedron $\Gamma_{+}(\sigma)$ has a single compact face of dimension 2. It corresponds to the convex hull of  $\{(1,0,0), (0,1,0), (1,1,2)\}$. Then $G(\sigma) \neq \Gamma(\sigma) \cap N$, which implies that  $\sigma$  is not $G$-stable.
\end{example}
 
\begin{example}\label{ex:no-c2}
In this example we exhibit a cone satisfying \cref{def:DCMGS}~$(i)$ but not $(ii)$. Let $\sigma=\con{(1,0,0), (0,1,0), (1,1,2), (0,0,-1)}\subset N_{\R}$. The minimal generating set of the corresponding semigroup is $G(\sigma)=\{(1,0,0), (0,1,0), (1,1,2), (0,0,-1)\}$. The polyhedron $\Gamma_{+}(\sigma)$ possesses two compact faces, denoted as $F_1$ and $F_2$, of dimension 2. These faces are defined as the convex hull of the sets $\{(1,0,0), (1,1,2), (0,0,-1)\}$ and $\{(0,1,0), (1,1,2), (0,0,-1)\}$, respectively. Then $G(\sigma)=\Gamma(\sigma)\cap N$. 

Let $A:=\{(1,0,0), (0,1,0), (1,1,2)\}$, and consider the cone $\sigma_A$, generated by $A$. Observe that $G(\sigma_A)=A\cup\{(1,1,1)\}$.  
Hence, \cref{def:DCMGS}~$(ii)$ is not satisfied.
\end{example} 

\begin{example}\label{ex:polygon}
Let $P\subset\R^2$ be a polygon such that all of its vertices belong to $\Z^2$. Such a polygon is called a lattice polygone. Consider the following cone, $\omega_{P}=\con{(v,1)\mid v\in P}\subset\R^2\times\R$. We claim that $\omega_{P}$ is $G$-stable. Indeed, every lattice polygon can be triangulated in such a way that each triangle $T$ has as lattice points only its vertices \cite[Proposition 1.1]{haase2021existence}.
For each such $T$ we have that the three-dimensional cone $\sigma_T$ is regular. 
It follows that $G(\omega_{P})$ is contained in the plane $\R^2\times\{1\}$. Any subset of $G(\omega_{P})$ gives place to a cone of the form $ \sigma_{Q} $ with $ Q  $ a lattice polygon, hence the same argument applies for this subset and \cref{def:DCMGS}~$(ii)$ holds.
\end{example}

\begin{remark}\label{rem:empty_false}
The fact that if $T$ has no integer points other than its vertices, then the cone $\sigma_T$ is regular is only true in dimension 2. For example, $ T = \operatorname{Conv} \left\{  (1,0,0), (0,1,0), (0,0,1), (1,1,1)  \right\}$ contains no integral points other than its vertices, and $\sigma_T$ is not regular. Indeed, its Hilbert basis contains the primitive vectors of the four rays of $\omega_P$ plus the extra element $(1,1,1,2)$.
\end{remark}

Let us now prove some basic properties of $G$-stable cones.

\begin{proposition}\label{pr:sub-dcmgs}
Let $\sigma\subset N_{\R}$ be a $G$-stable cone and let $A\subset G(\sigma)$. Then  $\Gamma_{+}(\sigma_A)=\Gamma_{+}(\sigma)\cap \sigma_{A}$. Furthermore,   $\sigma_A$ is $G$-stable.
\end{proposition}
\begin{proof}
Let $\sigma\subset N_\R$ be a cone. We prove the first statement by induction on the dimension of $\sigma_A$. If $\dim \sigma_A=1$, then $A=\{\gamma\}\subset G(\sigma)$ and $\sigma_{A}=\operatorname{Cone}(\gamma)=\R_{\geq 0}\gamma$. In this case, $\Gamma_{+}(\sigma_A)=\Gamma_{+}(\sigma)\cap \sigma_{A}=\R_{\geq 1}\gamma$.

Let now $1\leq d_0< \dim\sigma$ and assume that for every subset $A\subseteq G(\sigma)$  with $\dim\sigma_{A}<d_0$ we have that $\Gamma_{+}(\sigma_A)=\Gamma_{+}(\sigma)\cap \sigma_{A}$. 
Let $A\subset G(\sigma)$ such that $\dim \sigma_{A} =d_0$. Observe that
\begin{align} \label{eq:1}
  \Gamma_{+}(\sigma)\cap \sigma_{A}=\operatorname{Conv}\left ( (G(\sigma)\cap \sigma_A^{0})\cup \bigcup_{\tau\prec \sigma_A} (\Gamma_{+}(\sigma)\cap \tau)\right )\,,
\end{align}
where $\sigma_A^0$ denotes the relative interior and $\tau\prec \sigma_A$ runs over all proper faces. Since $A$ is a generating set of $\sigma_A$, for every proper face $\tau\prec\sigma$ there exists  $A_\tau\subsetneq A$ such that $ \tau=\sigma_{A_\tau}$. Hence, by the induction hypothesis we have 
$\Gamma_{+}(\sigma)\cap \tau=\Gamma_{+}(\tau)$. Furthermore, intersecting with $\sigma_A^0$ the equality of \cref{def:DCMGS}~$(ii)$ we obtain  that $G(\sigma)\cap \sigma_A^{0}=G(\sigma_A)\cap \sigma_A^{0}$. Replacing both these last equalities  in \eqref{eq:1}, we obtain
$$\Gamma_{+}(\sigma)\cap \sigma_{A}=\operatorname{Conv} \left ( (G(\sigma_A)\cap \sigma_A^{0})\cup\bigcup_{\tau\prec \sigma_A} \Gamma_{+}(\tau)\right )=\Gamma_{+}(\sigma_{A})\,.$$
This proves the first statement of the proposition.

\medskip

We now prove that $\sigma_{A}$ is $G$-stable. From the first statement of the proposition, it follows directly that  $\Gamma(\sigma_A)=\Gamma(\sigma)\cap \sigma_{A}$. Hence, by \cref{def:DCMGS}~$(ii)$, we obtain
$$G(\sigma_A)=G(\sigma)\cap\sigma_A=\Gamma(\sigma)\cap N\cap\sigma_A =\Gamma(\sigma_A)\cap N,$$
which proves \cref{def:DCMGS}~$(i)$. Let now $A'\subset G(\sigma_A)\subset G(\sigma)$. Since $\sigma$ is $G$-stable, we have that
$$G(\sigma_{A'})=G(\sigma)\cap\sigma_{A'}=G(\sigma)\cap \sigma_{A}\cap \sigma_{A'}=G(\sigma_A)\cap \sigma_{A'}\,,$$
which proves \cref{def:DCMGS}~$(ii)$ concluding the proof.
\end{proof}

\begin{corollary}\label{co:sub-dcmgs}
If $\sigma\subset N_{\R}$ is a $G$-stable cone and $\sigma'\preceq \sigma$  is a face, then $\sigma'$ is a $G$-stable cone.
\end{corollary}

We now prove that the $G$-stability of fans is preserved under certain refinements.

\begin{proposition}\label{pr:ref-Gstable}
Let $\Sigma$ be a $G$-stable fan and $\Sigma'$ a refinement of $\Sigma$ such that $\Sigma'(1)\subset G(\Sigma)$. Then $\Sigma'$ is $G$-stable and $G(\Sigma')=G(\Sigma)$.    
\end{proposition}

\begin{proof}
We first prove that $\Sigma'$ is $G$-stable, i.e., that every cone in $\Sigma'$ is $G$-stable. Let $\sigma'\in \Sigma'$. Since $\Sigma'$ is a refinement of $\Sigma$, there exists $\sigma\in \Sigma$ with $\sigma'\subset \sigma$. Moreover, the primitive vectors $\sigma'(1)$ of the rays of $\sigma'$ are contained in $G(\Sigma)$ and so in $G(\sigma)$. Since $\sigma'=\operatorname{Cone}(\sigma'(1))$ we have that $\sigma'$ is $G$-stable by \cref{pr:sub-dcmgs}.

We now prove that $G(\Sigma')=G(\Sigma)$. By \cref{def:DCMGS}~$(ii)$ we have $G(\sigma')= G(\sigma)\cap \sigma'$ for each $\sigma'\in\Sigma'$ and $\sigma\in\Sigma$ such that $\sigma'\subset\sigma$. Hence,
$$\bigcup_{\sigma'\in \Sigma'}G(\sigma')= \bigcup_{\sigma \in \Sigma}\left (\bigcup_{\sigma'\subset \sigma}G(\sigma')\right )=\bigcup_{\sigma \in \Sigma}\left (\bigcup_{\sigma'\subset \sigma}(G(\sigma)\cap\sigma')\right )=\bigcup_{\sigma\in \Sigma}G(\sigma)\,.$$
This proves the lemma.
\end{proof}

\begin{corollary}
Let $\Sigma_1$  be a $G$-stable fan, and consider a chain of refinements:
$$\Sigma_m\prec_\text{\rm ref} \dots \prec_\text{\rm ref} \Sigma_1, \quad\mbox{with}\quad \Sigma_m(1)\subset G(\Sigma_1)\,.$$
Then the fans $\Sigma_i$ are $G$-stable, and $G(\Sigma_i)=G(\Sigma_1)$ for all $1\leq i\leq m$.
\end{corollary}
\begin{proof}
Since $\Sigma_1(1)\subset \Sigma_2(1)\subset \dots \subset  \Sigma_m(1)\subset G(\Sigma_1)$, the corollary follows from \cref{pr:ref-Gstable}.
\end{proof}

\subsection{G-desingularization}\label{subsec:G-desingularization}

The aim of this section is to prove the following theorem.

\begin{theorem}\label{th:G-des}
Let $\Sigma$ be a $G$-stable fan in $N_{\R}$. Then $X(\Sigma)$ admits a $G$-desin\-gu\-lar\-iza\-tion.
\end{theorem} 

The general strategy to prove this result is the following. Firstly, using simplicial subdivisions and \cref{pr:ref-Gstable}, we can assume that the fan is simplicial. Let now $\Sigma$ be a simplicial $G$-stable fan of dimension $n=\operatorname{rank} N$. We will show that there exists a chain of simplicial refinements
\begin{align} \label{eq:refinement}
 \Sigma_n \prec_\textrm{ref} \dots \prec_\textrm{ref} \Sigma_i\prec_\textrm{ref} \dots  \prec_\textrm{ref} \Sigma_1=\Sigma\,,
\end{align}
such that for each  $1\leq i\leq n$ we have:
\begin{enumerate}
\item[$(i)$] $\Sigma_i(1)\subseteq G(\Sigma)$.
\item[$(ii)$] Every $\sigma\in \skel_i(\Sigma_i)$ is regular.
\item[$(iii)$] If $i>1$ and $\sigma\in \skel_i(\Sigma_{i-1})$ is regular, then $\sigma\in  \skel_i(\Sigma_{i})$.
\end{enumerate}
The last condition ensures that regular cones of dimensions smaller or equal to $i$ are not refined. Observe that $\Sigma_n$ is indeed regular. This will provide that the $G$-stable fan $\Sigma$ admits a regular refinement $\Sigma_n$ such that $\Sigma_n(1)=G(\Sigma)$, yielding that $X(\Sigma_n)\to X(\Sigma)$ is a $G$-desingularization. In the remaining of this section we construct  the refinement  \eqref{eq:refinement}.

\medskip

Let us define certain numerical quantities that will intervene in our proof of existence of the refinement \eqref{eq:refinement}. Let $\Sigma$ be a simplical fan in $N_\R$ and let $i\geq 1$ be an integer. Consider the following numbers ($\#A$ denotes the cardinality of a set $A$):
\begin{align*}
M(i,\Sigma)&=\max\big\{ \#G(\sigma) \mid \sigma\in \Sigma(i) \big\}-i, \\
N(i, \Sigma)&=\#\left\{\sigma\in \Sigma(i)\mid \#G(\sigma)-i\geq \max\{M(i,\Sigma),1\}\right\}.
\end{align*}
As a simple example, notice that $M(1,\Sigma)=0$ and $N(1,\Sigma)=0$.

\begin{remark}\label{rem: M()-N()}
Since $\Sigma$ is a simplicial fan, for each cone $\sigma \in \Sigma$, we have $\#\sigma(1) = \dim \sigma$. Additionally, $\sigma \in \Sigma$ is regular if and only if $\sigma(1) = G(\sigma)$. Using these two facts it follows directly that all cones of $\Sigma$ of dimension $i$ are regular if and only if $M(i, \Sigma) = 0$. Similarly, all cones of $\Sigma$ of dimension $i$ are regular if and only if $N(i, \Sigma) = 0$. Furthermore, whenever $M(i,\Sigma)\geq 1$, the quantity $N(i,\Sigma)$ counts the number of cones $\sigma\in \Sigma$ realizing the maximum in the definition of $M(i,\Sigma)$.
\end{remark}

Fix $n\in\N_{\geq1}$. Let $\Omega=\{G\mbox{-stable simplicial fans of dimension }n \mbox{ in }N_{\R}\}$. Let $i_0$ be the largest integer which is smaller or equal than $n$ and such that for all $\Sigma\in\Omega$, there exists a sequence of length $i_0$ of simplicial refinements
\begin{align} \label{eq:refinement2}
 \Sigma_{i_0} \prec_\textrm{ref} \dots \prec_\textrm{ref} \Sigma_i\prec_\textrm{ref} \dots  \prec_\textrm{ref} \Sigma_1=\Sigma\,,
\end{align}
such that for each  $1\leq i\leq i_0$ we have:
\begin{enumerate}
\item[$(i)$] $\Sigma_i(1)\subseteq G(\Sigma)$.
\item[$(ii)$] Every $\sigma\in \skel_i(\Sigma_i)$ is regular.
\item[$(iii)$] If $i>1$ and $\sigma\in \skel_i(\Sigma_{i-1})$ is regular, then $\sigma\in  \skel_i(\Sigma_{i})$.
\end{enumerate}
Notice that $i_0\geq 1$ since $\Sigma_1=\Sigma$ satisfies $(i)$, $(ii)$ and $(iii)$ above. Notice also that $\Sigma_i$ is $G$-stable for each $1\leq i\leq i_0$ by \cref{pr:ref-Gstable}.

The following is the main technical lemma that will intervene in the proof of \cref{th:G-des}.

\begin{lemma} \label{lm:induction}
Let $\Sigma$ be a $G$-stable simplicial fan of dimension $n$ and let $i\leq \min \{i_0,n-1\}$ be such that all cones of $\skel_{i}(\Sigma)$ are regular and  $M(i+1,\Sigma)>0$. Then there exists a simplicial refinement $\Sigma'$ of $\Sigma$ satisfying the following properties: 
\begin{enumerate}
\item[$(a)$] $\Sigma'(1)\subseteq G(\Sigma)$. In particular, $\Sigma'$ is $G$-stable.
\item[$(b)$] All cones in $\skel_{i}(\Sigma')$  are regular.
\item[$(c)$] If $\sigma\in \skel_{i+1}(\Sigma)$ is regular, then $\sigma\in \skel_{i+1}(\Sigma')$.
\item[$(d)$] $(M(i+1,\Sigma'), N(i+1,\Sigma'))<_{\textrm{lex}}(M(i+1,\Sigma), N(i+1,\Sigma))$, where $<_\textrm{lex}$ is the lexicographic order.
\end{enumerate}
\end{lemma}
\begin{proof}
Let $\sigma\in \Sigma(i+1)$ be a cone realizing the maximum in the definition of $M(i+1,\Sigma)$, i.e., a cone such that $\#G(\sigma)=M(i+1, \Sigma)+i+1>i+1$. Hence, $\sigma$ is a simplicial but non-regular cone. This yields that $G(\sigma)\setminus \Sigma(1)\neq \emptyset$. Let $\tau\prec\sigma$ be a proper face of $\sigma$. The cone $\tau$ is regular since, by hypothesis, all cones in $\skel_i(\Sigma)$ are regular. Then $G(\sigma)\setminus \sigma(1)$ is contained in the relative interior $\sigma^0$ of $\sigma$. Letting now $\gamma\in G(\sigma)\setminus\sigma(1)$ we let $\widehat{\Sigma}=\Sigma(\gamma)$ be the refinement given by the star subdivision of the fan $\Sigma$ at $\gamma$ \cite[Lemma~11.1.3]{cox2011toric}. In particular, $\widehat{\Sigma}$ is simplicial.

Remark first that  $\widehat{\Sigma}(1)=\Sigma(1)\cup\{\gamma\}\subset G(\Sigma)$. Hence, by \cref{pr:ref-Gstable}, we conclude that $\widehat{\Sigma}$ is $G$-stable. Moreover, since $\gamma$ is contained in the relative interior of the cone $\sigma$, every cone in $\Sigma$ different from $\sigma$ remains unchanged in $\widehat{\Sigma}$. In particular, $\widehat{\Sigma}\prec_\textrm{ref}\Sigma$ does not refine regular cones in $\Sigma$. However, the refinements of $\sigma$ introduced by the star subdivision may produce simplicial non-regular cones of dimension less or equal than $i$ in $\widehat{\Sigma}$. Since $i\leq i_0$, by the definition of $i_0$, there exists a chain of $G$-stable simplicial refinements of $\widehat{\Sigma}$ given by
$$\widehat{\Sigma}_i\prec_\text{\rm ref} \dots \prec_\text{\rm ref} \widehat{\Sigma}_j \prec_\text{\rm ref}\dots \prec_\text{\rm ref} \widehat{\Sigma}_{1}=\widehat{\Sigma}\,,$$ 
satisfying   $(i)$, $(ii)$ and $(iii)$ in the definition of $i_0$ in \eqref{eq:refinement2}. Moreover, we can impose that only cones with support in $\sigma$ are refined in the above chain.

Let $\Sigma'=\widehat{\Sigma}_i$. Now,  $\Sigma'\prec_\textrm{ref}\Sigma$ is the refinement claimed to exist in the lemma. 
By $(i)$ we have that $\Sigma'(1)\subseteq G(\widehat{\Sigma})=G(\Sigma)$, where the last equality follows from \cref{pr:ref-Gstable}. This yields $(a)$. By $(ii)$ we have that $(b)$ holds. As remarked above, no smooth cone of $\Sigma$ is refined at any step implying  $(c)$. 

Let us prove $(d)$. Firstly, notice that for every cone $\sigma'\in \Sigma'(i+1)$ with $\sigma'\subsetneq\sigma$ we have that $G(\sigma')=G(\sigma)\cap \sigma'$ by the $G$-stability of $\Sigma$. By the refinement made to $\sigma$ it follows that $\#G(\sigma')<\#G(\sigma)$. Suppose that $N(i+1,\Sigma)\geq 2$, i.e., there are at least two $(i+1)$-dimensional cones in $\Sigma$ realizing the maximum in the definition of $M(i+1,\Sigma)$. This implies that $M(i+1,\Sigma)=M(i+1,\Sigma')$ and $N(i+1,\Sigma')=N(i+1,\Sigma)-1$. Hence, $(M(i+1,\Sigma'),N(i+1,\Sigma'))<_{lex}(M(i+1,\Sigma),N(i+1,\Sigma))$. On the other hand, if $N(i+1,\Sigma)=1$ then there are no cones left realizing the maximum and so $M(i+1,\Sigma')<M(i+1,\Sigma)$. This yields $(d)$.
\end{proof}

\begin{proof}[Proof of \cref{th:G-des}]
Let $\Sigma$ be a $G$-stable fan of dimension $n$. If $\Sigma$ is not simplicial, there exists a simplicial refinement $\widehat{\Sigma}$ with $\widehat{\Sigma}(1)=\Sigma(1)$ \cite[Proposition 11.1.7]{cox2011toric}. Hence, $\widehat{\Sigma}$ is also $G$-stable and $G(\Sigma)=G(\widehat{\Sigma})$ by \cref{pr:ref-Gstable}. Since a $G$-desingularization of $\widehat{\Sigma}$ provides a $G$-desingularization of $\Sigma$, without loss of generality we assume that $\Sigma$ is simplicial in the sequel.

To continue with the proof, recall the definition of $i_0$ given in \eqref{eq:refinement2}. In what follows we assume that $\Sigma$ is a $G$-stable simplicial fan whose longest sequence of refinements satisfying \eqref{eq:refinement2} is of length $i_0$. We are going to show that $i_0=n$ for this particular $\Sigma$. By definition of $i_0$, this provides the existence of a $G$-desingularization of every $G$-stable simplicial fan of dimension $n$.

Assume $i_0<n$ and let 
\begin{align*}
 \Sigma_{i_0} \prec_\textrm{ref} \dots \prec_\textrm{ref} \Sigma_i\prec_\textrm{ref} \dots  \prec_\textrm{ref} \Sigma_1=\Sigma\,,
\end{align*}
be a sequence of refinements as in $\eqref{eq:refinement2}$. If  $M(i_0+1,\Sigma_{i_0})=0$ then every $(i_0+1)$-dimensional cone of $\Sigma_{i_0}$ is regular and setting  $\Sigma_{i_0+1}=\Sigma_{i_0}$ we have that the sequence 
\begin{align}\label{eq:contradiction}
 \Sigma_{i_0+1} \prec_\textrm{ref} \Sigma_{i_0} \prec_\textrm{ref} \dots \prec_\textrm{ref} \Sigma_i\prec_\textrm{ref} \dots  \prec_\textrm{ref} \Sigma_1=\Sigma\,,
\end{align}
satisfies  conditions $(i)$, $(ii)$, and $(iii)$, which contradicts the maximality of $i_0$. 

This yields that $M(i_0+1,\Sigma_{i_0})>0$. Let now $\Sigma'_1=\Sigma_{i_0}$. By applying \cref{lm:induction} inductively, we can construct a sequence of $G$-stable refinements
\begin{align*}
 \Sigma'_{m}  \prec_\textrm{ref} \dots \prec_\textrm{ref} \Sigma'_j\prec_\textrm{ref} \dots  \prec_\textrm{ref} \Sigma'_1\,,
\end{align*}
such that 
$$(M(i_0+1,\Sigma'_{j+1}), N(i_0+1,\Sigma'_{j+1}))<_{\textrm{lex}}(M(i_0+1,\Sigma'_j), N(i+1,\Sigma'_j)).$$
Since $(M(i_0+1,\Sigma'_{j+1}), N(i_0+1,\Sigma'_{j+1}))$ belongs to $\Z_{\geq 0}^2$ which is bounded below by $(0,0)$ with the lexicographical order, this chain can be made to end with $\Sigma'_{m}$ satisfying $M(i_0+1,\Sigma'_m)=0$ and so setting $\Sigma_{i_0+1}=\Sigma'_{m}$ in \eqref{eq:contradiction} provides again a contradiction to the maximality of $i_0$. This yields $i_0=n$ proving the theorem.
\end{proof}

\begin{remark}
In addition to being a key step for the proof of one of the main theorems in this paper, \cref{th:G-des} has an interest on its own since it is known that $G$-desingularizations do not exist in general in dimension higher than three \cite{BoGS95} (see also \cite[Theorem 3.12]{KY}). The existence of $G$-desingularizations is important in the study of a special kind of divisors associated to desingularizations of toric varieties \cite{Bouvier,ChDY}.
\end{remark}

\begin{remark}\label{rem:moderate}
Let $\sigma\subset N_{\R}$ be a full-dimensional cone. A \textit{moderate toric resolution} of $\sigma$ is a regular subdivision $\Sigma$ of $\sigma$ such that for each cone $\tau\in\Sigma$, the affine hyperplane spanned by the ray generators of $\tau$ intersects every ray of $\sigma$. This notion was introduced in \cite{ChDY} as a tool to study essential divisors of $F$-blowups of toric varieties. Let us briefly comment on the relation of moderate toric resolutions, $G$-desingularizations and $G$-stability.

Let $\sigma\subset N_{\R}$ be a $G$-stable cone and let $\Sigma$ be a $G$-desingularization. Let $\tau\in\Sigma$ and let $H_{\tau}$ be the affine hyperplane generated by the ray generators of $\tau$. By \cref{pr:sub-dcmgs}, $\tau$ is $G$-stable. In particular, $\Gamma(\tau)\subset\Gamma(\sigma)$. Since $\tau$ is regular we have that $\Gamma(\tau)=\tau\cap H_{\tau}$. We conclude that $\Sigma$ is a moderate toric resolution by \cite[Proposition 6.10]{ChDY}.

Hence, we see that $G$-stability is an important notion both for Nash blowups and $F$-blowups.
\end{remark}

\subsection{The normalized Nash blowup of a normal toric variety}\label{subsec:Nash free of char}

In this section we relate the notion of $G$-stability of a cone $\sv$ with the normalized Nash blowup of $X(\sigma)$. We prove that for $G$-stable cones, the combinatorics of the normalized Nash blowup of $X(\sigma)$ is independent of the characteristic of the base field. In what follows we consider $G$-stable cones in $M_{\R}$, the definition of $G$-stability being exactly the same as for cones in $N_{\R}$.

Let $\sigma\subset N_{\R}$ be a $d$-dimensional cone, where $d=\operatorname{rank}(N)$. Let $\sv\subset M_{\R}$ be its dual cone and recall that $\gsv$ is the minimal generating set of the semigroup $\sv\cap M$. Let $\xs$ be the corresponding normal toric variety.

For a collection of $d$ elements $\{h_1,\dots,h_d\}\subset\gsv$, we define the matrix $(h_1 \cdots h_d)$ whose columns are the vectors $h_i$. Letting $p\geq0$ be the characteristic of the base field $\K$ we denote 
$$\operatorname{det}_p(h_1\cdots h_d)=
\begin{cases}
  \det(h_1\cdots h_d) & \mbox{if } p=0 \\
  \det(h_1\cdots h_d) \mod p & \mbox{if } p>0
\end{cases}\,.$$

Consider the following polyhedron associated to the toric variety $X(\sigma)$:
\begin{equation}\label{eq:N_polyhedron}
     \mathcal{N}_p(\sigma)=\operatorname{Conv}\big\{(h_{1}+\dots+h_{d})+\sv\mid h_{i}\in G(\sigma^\vee)\mbox{ and }\operatorname{det}_p(h_1 \cdots h_d)\neq 0\big\}.
\end{equation}

As one may expect, this polyhedron depends on the characteristic of the field. Indeed, there are explicit examples of a cone $\sigma$ such that $\mathcal{N}_0(\sigma)\neq\mathcal{N}_p(\sigma)$ for some $p\neq0$ \cite[Section 3]{DJNB}.

In this section we show that the polyhedron $\mathcal{N}_p(\sigma)$ is independent of the characteristic $p$ if $\sigma^\vee$ is $G$-stable. As we will see at the end of the section, this result implies that the combinatorics of the normalized Nash blowup of $X(\sigma)$ are independent of the characteristic of the base field (see \cref{cor:Nash-free}).

\begin{theorem} \label{th:Nash polyh}
Let $\sigma \subset N_\R$ be a full-dimensional strongly convex cone. If $\sigma^\vee$ is $G$-stable then $\mathcal{N}_0(\sigma)=\mathcal{N}_p(\sigma)$ for all $p>0$ prime.
\end{theorem}

Before proceeding to the proof of \cref{th:Nash polyh}, we prove the following lemma.

\begin{lemma} \label{lem:convex-function}
Let $\sigma^\vee\subset M_\R$ be a simplicial full-dimensional $G$-stable cone. Then there exists a unique primitive vector $\gamma_0 \in N$, and a unique integer $\ell_0>0$ such that:
\begin{enumerate}
    \item[$(i)$] $\langle h, \gamma_0 \rangle = \ell_0$ for all $h\in \sigma^\vee(1)$, and
    \item[$(ii)$] $\langle h, \gamma_0 \rangle \leq \ell_0$ for all $h \in G(\sigma^\vee)$.
\end{enumerate}
\end{lemma}
\begin{proof}
The set $\sigma^\vee(1)$ is a basis of $M_\R$ and, since they are all elements in $M$, they also form a basis of $M\otimes_\Z \Q$. Hence, there exists a unique vector $\gamma \in N\otimes_\Z \Q$ such that $\langle h, \gamma \rangle = 1$ for all $h\in \sigma^\vee(1)$. Eliminating denominators, we obtain the existence of $\gamma_0$ and $\ell_0$ satisfying $(i)$ in the lemma and the uniqueness is assured by the fact that $\gamma_0$ must be primitive. By \cref{def:DCMGS}~$(i)$ we have that $G(\sigma^\vee)=\Gamma(\sigma^\vee)\cap M$. Hence, condition $(ii)$ in the lemma follows directly from the convexity of $\Gamma_+(\sigma^\vee)$.
\end{proof}

We procede now to the proof of \cref{th:Nash polyh}.

\begin{proof}[Proof of \cref{th:Nash polyh}]
Let $\mathcal{J}_p(\sigma) \subset M$ be the set consisting of all vectors of the form
$$\mathcal{J}_p(\sigma)=\left 
\{h_{1}+\dots+h_{d}\mid h_{i}\in G(\sigma^\vee)\mbox{ and }\operatorname{det}_p(h_1\cdots h_d)\neq 0\right\}\,.$$ 
By definition,  $\mathcal{N}_p(\sigma)=\operatorname{Conv}(\mathcal{J}_p(\sigma)+\sigma^\vee)$. Hence, every vertex of $\mathcal{N}_p(\sigma)$ is contained in $\mathcal{J}_p(\sigma)$. Each $v \in \mathcal{J}_p(\sigma)$ can be expressed in the form $v = h_1 + h_2 + \dots + h_d$, 
where $h_i \in G(\sigma^\vee)$. However, this representation may not be unique. For every $v \in \mathcal{J}_0(\sigma)$, we let 
$$
d(v) = \min\left\{\lvert\operatorname{det}_0(h_1\cdots h_d)\rvert \mid h_{i}\in G(\sigma^\vee),\, h_1+\dots+h_d=v,\, \mbox{ and } \dez(h_1\cdots h_d)\neq 0\right\}\,. 
$$

Notice that, by definition, $\mathcal{J}_p(\sigma)\subset \mathcal{J}_0(\sigma)$. Thus $\mathcal{N}_p(\sigma)\subset \mathcal{N}_0(\sigma)$. If $v\in\jz$ satisfies $d(v)=1$ then $v\in\jp$. In particular, if every vertex $v$ of $\mathcal{N}_0(\sigma)$ satisfies $d(v)=1$ then $v$ also belongs to $\mathcal{N}_p(\sigma)$ implying $\mathcal{N}_p(\sigma)=\mathcal{N}_0(\sigma)$.

To prove the theorem, we show the following: if $v \in \mathcal{J}_0(\sigma)$ is such that $d(v) > 1$, then $v$ is not a vertex of the polyhedron $\mathcal{N}_0(\sigma)$. 

Let $v \in \mathcal{J}_0(\sigma)$ and assume that $d(v) > 1$. We let  $h_1, \dots, h_d \in G(\sigma^\vee)$ be such that $d(v) = \lvert\dez(h_1\cdots h_d)\rvert$. Denote as $\omega\subset M_{\R}$ the cone generated by $\{h_1,\dots, h_d\}$. Since $|\dez(h_1 \cdots h_d)|>1$, the simplicial cone $\omega\subset M_\R$ is not regular. By \cref{pr:sub-dcmgs}, the cone $\omega$ is $G$-stable and by \cref{co:sub-dcmgs}, every face $\delta\prec\omega$ is also $G$-stable. 

In the rest of the proof we fix a special  $\delta\prec \omega$ to construct a polytope $P\subset \mathcal{N}_0(\sigma)$ such that $v\in P$ but it is not a vertex of $P$. Hence it cannot be a vertex of $\mathcal{N}_0(\sigma)$ either, proving the theorem.

Let $\delta$ be a face of $\omega$ that is not regular and such that all of its proper faces are regular. Since all cones of dimension 0 or 1 are regular, we have that $\dim\delta\geq 2$. Without loss of generality, we can assume that  $\delta$ is generated by $\{h_1,\dots,h_e\}$,  where $2 \leq e \leq d$. Since  $\{h_1,\dots,h_e\}$ is linearly independent in $M_\R$ we have $e=\dim \delta$. By \cref{th:G-des}, the cone $\delta$ admits a $G$-desingularization. 

Let $\Delta\prec_\textrm{ref} \delta$ be a $G$-desingularization of $\delta$, i.e., $\Delta$ is a regular refinement of $\delta$ such that $\Delta(1)=G(\delta)$. Now let $\delta'\in \Delta$ be the unique cone containing $h_1 + \dots + h_e$ in its relative interior. Since $\delta'$ is regular, we have $G(\delta')=\delta'(1)=\{h'_1,\dots,h'_l\}$ and so we have
\begin{align} \label{eq:h-h'}
h_1 + h_2+ \dots + h_e=a_1h'_1+a_2h'_2+\dots+a_lh'_l,\quad\mbox{where}\quad a_i\in\Z_{\geq1}\,.
\end{align}

\begin{claim*}
$h'_i \notin \{h_1, \dots, h_e\}$ for all $1 \leq i \leq l\,.$
\end{claim*}
\begin{proof}[Proof of the claim]
First we prove that there exists at least one $1 \leq i \leq l$ such that
\[
h'_i \notin \{h_1, h_2, \dots, h_e\}.
\]
We proceed by contradiction. Assume that $\{h'_1,h'_2,\dots,h'_l\} \subset \{h_1, h_2, \dots, h_e\}$.  If $l = e$ then $\delta' = \delta$ which is impossible since $\delta'$ is regular while $\delta$ is not. Hence $l < e$ and $\delta'$ is a proper face of $\delta$. We know that $h_1 + h_2 + \dots + h_e$ is in the relative interior of $\delta$. In particular, it is not contained in any proper face of $\delta$. This contradicts the choice of $\delta'$.

We now assume that a proper subset of $\{h'_1,h'_2,\dots,h'_l\}$  is contained in $\{h_1, h_2, \dots, h_e\}$. Without loss of generality, we can assume $h'_i=h_i$ for $1\leq i\leq r<l$, and $ h'_i\notin \{h_1,...,h_e\}$ for $r<i\leq l$. If $a_1 = 1$, then \eqref{eq:h-h'} yields
\[
h_2 + \dots + h_e = \underbrace{a_2 h_2 + \dots + a_r h_r + a_{r+1} h'_{r+1} + \dots + a_{l-1}h'_{l-1}}_{\alpha} + \underbrace{a_l h'_l}_{\beta}.
\]

Now, $\alpha,\beta\in \delta$ and $\alpha+\beta\in \operatorname{Cone}(h_2, \dots,h_e)=:\delta_0$. Since $\delta_0\prec \delta$ is a proper face, it follows that $\alpha,\beta\in \delta_0$ and so $h'_l\in \delta_0$. Moreover, $h'_l\in \delta'(1)\subset\Delta(1)=G(\delta)$. Hence $h'_l\in G(\delta)\cap\delta_0$. Since $\delta$ is $G$-stable it follows that $h'_l\in G(\delta_0)$. We also know that $h'_l\notin \delta_0(1)=\{h_2, \dots,h_e\}$. This implies that $\delta_0$ is a singular cone. This contradicts the choice of $\delta$ which assures that each of its proper faces are regular.

If $a_1 \geq 2$, then \eqref{eq:h-h'} yields
\[
h_2 + \dots + h_e = \underbrace{(a_1 - 1) h_1}_\alpha + \underbrace{a_2 h_2 + \dots + a_r h_r + a_{r+1} h'_{r+1} + \dots + a_l h'_l}_\beta.
\]
Now, $\alpha,\beta\in \delta$ and $\alpha\notin\delta_0$. Hence, $\alpha+\beta\notin\delta_0$ which is a contradiction since $h_2 + \dots + h_e\in \delta_0$. This concludes the proof of the claim. \phantom\qedhere
\end{proof}

We now apply the claim to show that for every $i\in\{1,\dots,e\}\mbox{ and } j\in\{1,\dots,l\}$, the set
\begin{align} \label{eq:li}   
\{h_1, \dots, h_{i-1}, \hat{h}_i, h_{i+1} \dots,  h_d\} \cup \{h'_j\} \mbox{ is linearly independent,}
\end{align}
where $\hat{h_i}$ indicates that this element is removed from the list. Indeed, by the claim we know $h'_j\notin\{h_1,h_2,\dots,h_e\}=\delta(1)$. Since all faces of $\delta$ are regular by definition, we have that $h'_j$ is in the relative interior of $\delta$. Hence,
\[
h'_j = c_1 h_1 + c_2 h_2 + \dots + c_e h_e,\quad\mbox{with}\quad c_i> 0.
\]
This proves \eqref{eq:li} since $\{h_1, \ldots, h_d\}$ is linearly independent.

Now, let $\gamma_0$ and $\ell_0$ be as in \cref{lem:convex-function} applied to the cone $\delta=\operatorname{Cone}(h_1,\dots,h_e)$. Recalling that $h'_j\in G(\delta)$ for all $j$, the lemma and \eqref{eq:h-h'} imply
$$e\ell_0=\langle h_1 + h_2+ \cdots + h_e,\gamma_0\rangle=\langle a_1h'_1+a_2h'_2+\dots+a_lh'_l,\gamma_0\rangle\leq \ell_0\sum_{j=1}^l a_j\,.
$$
Thus $e\leq \sum_{j=1}^l a_j$. This provides that there exists a function $f: \{1, 2,\ldots, e\} \to \{1,\ldots, l\}$ such that
$$
1\leq\#f^{-1}(j) \leq a_j, \quad\mbox{for all}\quad 1 \leq j \leq l\,.
$$
These last inequalities also assure that we can pick positive integers
 $b_1, b_2,\dots, b_e$ such that 
\begin{align} \label{eq:los-bi}
    \sum_{i \in f^{-1}(j)} b_i = a_j, \quad\mbox{for all}\quad \quad 1 \leq j \leq l\, .
\end{align}
For each $i\in\{1,\dots,e\}$, Equation \eqref{eq:li} implies that the set $\{h_1,\dots,h_{i-1}, \hat{h}_i, h_{i+1},\dots, h_d\}\cup\left\{h'_{f(i)}\right\}$ is linearly independent, for each $i\in\{1,\dots,e\}$. 
Hence $h_1 + \dots +h_{i-1}+ \hat{h}_i+ h_{i+1}+\dots+ h_d+ h'_{f(i)}\in \mathcal{J}_0(\sigma)$ (as before $\hat{h}_i$ does not appear in the sum). Recalling that $h'_j\in\delta'\subset\delta\subset\omega\subset\sv$ we obtain
\begin{align} \label{eq:elemento-raro}
h_1 + \cdots& +h_{i-1}+\hat{h}_i+ h_{i+1}+\dots+ h_d+ b_i h'_{f(i)}=\notag\\
&(h_1 + \dots +h_{i-1}+ \hat{h}_i+ h_{i+1}+\dots+ h_d+ h'_{f(i)})+(b_i-1)h'_{f(i)}\in \mathcal{N}_0(\sigma)\,.
\end{align}
Applying \eqref{eq:h-h'} and \eqref{eq:los-bi}, we have that 
\begin{equation}\label{eq:la-milagrosa}
  \begin{split}
\sum_{i=1}^{e} (h_1 + &\dots +h_{i-1}+\hat{h}_i+ h_{i+1}+\dots+ h_d+ b_i h'_{f(i)}) \\
&=(e-1)(h_1+\dots+h_e)+e(h_{e+1}+\cdots+h_d)+\sum_{i=1}^e b_ih'_{f(i)} \\
&=(e-1)(h_1+\dots+h_e)+e(h_{e+1}+\cdots+h_d)+\sum_{j=1}^l \sum_{i\in f^{-1}(j)}b_ih'_j \\
&=(e-1)(h_1+\dots+h_e)+e(h_{e+1}+\cdots+h_d)+(a_1h'_1+a_2h'_2+\dots+a_lh'_l) \\
&=e(h_1+\dots+h_e)+e(h_{e+1}+\cdots+h_d)\\
&=e(h_1+\dots+h_d).
  \end{split}
\end{equation}

Assume for a moment that 
$h_1 + \dots + \hat{h}_i + \dots + h_d + b_i h'_{f(i)} = h_1 + \dots + \hat{h}_j + \dots + h_d + b_{j} h'_{f(j)}$ for all $i,j\in\{1,\ldots,e\}$. Replacing this identity in \eqref{eq:la-milagrosa} we obtain
\[
h_1 + \dots + \hat{h}_i + \dots + h_d + b_i h'_{f(i)}=h_1 + \dots + h_d\quad\mbox{for all } \quad i\in\{1,\dots,e\}.
\]
This yields $h_i = b_i h'_{f(i)}$ for each $i$. The elements $h_i$ and $h'_j$ belong to $G(\sigma^\vee)$ since $\sigma^\vee$ is $G$-stable. Hence $h_i$ and $h'_j$ are primitive vectors. Thus $b_i=1$ and $h_i = h'_{f(i)}$. This contradicts the claim. 

We conclude that there exist $i,j \in \{1, \dots, e\}$ such that
\[
h_1 + \dots + \hat{h}_i + \dots + h_d + b_i h'_{f(i)} \neq h_1 + \dots + \hat{h}_j + \dots + h_d + b_{j} h'_{f(j)}.
\]
This yields that the polytope 
\[
P=\operatorname{Conv}\left(\left\{ (h_1 + \dots + \hat{h}_i + \dots + h_d + b_i h'_{f(i)}) : 1 \leq i \leq e \right\}\right)
\]
is not a single point. Furthermore, by \eqref{eq:elemento-raro} we have that $P\subset \mathcal{N}_0(\sigma)$. Finally, \eqref{eq:la-milagrosa} implies that  $v=h_1 + \dots + h_d$ is not a vertex of $P$ and so $v$ is not a vertex of  $\mathcal{N}_0(\sigma)$, concluding the proof.
\end{proof}

Let us now explain the consequences of  \cref{th:Nash polyh} on the normalized Nash blowup of a normal toric variety.

Given a full-dimensional lattice polyhedron $P\subset M_{\R}$, i.e., a polyhedron of dimension $\operatorname{rank} M$ having its vertices in $M$, we define its \emph{normal fan} in $N_{\R}$ as the unique fan whose maximal cones are $\sigma_v$, where $v$ runs over all the vertices of $P$, and  $\sigma_v$ is the dual cone of the cone in $M_{\R}$ generated by
\vspace{-1mm}
\begin{align*} 
\{w-v\in M_{\R} \mid w \mbox{ is a vertex of } P \}.
\end{align*}

The following theorem shows the relation of the polyhedron $\mathcal{N}_p(\sigma)$ and the Nash blowup of $X(\sigma)$ (see \cite[Propositions 32 and 60]{GoTe14} for the characteristic zero case and \cite[Proposition 32]{GoTe14} and \cite[Theorem 1.9]{DJNB} for the prime characteristic case).

\begin{theorem} \label{th:Nash-blowup}
The normalized Nash blowup of the toric variety $X(\sigma)$ is $X(\Sigma_p)\to X(\sigma)$ where $\Sigma_p$ is the normal fan of the polyhedron $\mathcal{N}_p(\sigma)$.
\end{theorem}

In other words, \cref{th:Nash-blowup} states that the normalized Nash blowup of $X(\sigma)$ depends completely on the polyhedron $\mathcal{N}_p(\sigma)$. Hence, \cref{th:Nash polyh} has the following consequence.

\begin{corollary}\label{cor:Nash-free}
Let $\sigma \subset N_\R$ be a full-dimensional strongly convex cone and such that $\sigma^\vee$ is $G$-stable. Let $\mathcal P$ be any property of $X(\Sigma_p)$ that corresponds to a combinatorial property of $\Sigma_p$. Then $\mathcal P$ holds for $X(\Sigma_0)$ if and only if $\mathcal P$ holds for $X(\Sigma_p)$ for all $p>0$. In particular, $X(\Sigma_0)$ is non-singular if and only if $X(\Sigma_p)$ is non-singular.
\end{corollary}

We conclude with an example illustrating the objects and results of this section. In the example we will find:
\begin{enumerate}
    \item A $G$-stable cone whose Hilbert basis is contained in a hyperplane. 
    \item The normalized Nash blowup of the variety of such cones can be singular.
    \item The $G$-stability is not preserved in the polyhedron defining the normalized Nash blowup.
\end{enumerate}

\begin{example}\label{ex:thms-1}
Let $\sigma^\vee\subset M_\mathbb{R}$ be the cone generated by the columns of the matrix
$$
B=\left[\begin{array}{rrr}
1 & 0 & 2 \\
0 & 1 & 2  \\
0 & 0 & 3  
\end{array}\right]\,.
$$
The minimal generating set of the semigroup $\sigma^\vee\cap M$ is 
$G(\sigma^\vee)=\{h_1,\ldots,h_4\}$, where the first three elements are the ray generators of the cone and $h_4=(1,1,1)$. The cone $\sigma^\vee$ is $G$-stable. Indeed, $G(\sigma^\vee)$ is contained in the hyperplane given by $z_1+z_2-z_3=1$. Hence, \cref{def:DCMGS}~$(i)$ holds. Moreover, the cone generated by every proper subset of $G(\sigma^\vee)$ is regular. This yields \cref{def:DCMGS}~$(ii)$.

Let $\sigma\subset N_\mathbb{R}$ be the cone whose dual is $\sigma^\vee$. Recall that $p\geq0$ denotes the characteristic of the field $\K$. By \cref{th:Nash-blowup}, the Nash blowup of $X(\sigma)$ is $X(\Sigma_p)$ where $\Sigma_p$ is the normal fan of $\mathcal{N}_p(\sigma)$. By \cref{th:Nash polyh}, we have that  $\mathcal{N}_0(\sigma)=\mathcal{N}_p(\sigma)$. Let us illustrate this fact by doing the explicit computations in this particular example. Firstly,
\begin{align} \label{eq:exthms-1}
\det(h_1\,h_2\,h_4)=\det(h_1\,h_3\,h_4)=\det(h_2\,h_3\,h_4)=\pm1\quad \mbox{and}\quad\det(h_1\,h_2\,h_3)=3\,.
\end{align}
Let $v_i=(h_1+h_2+h_3+h_4)-h_i$, $1\leq i\leq 4$. By \eqref{eq:exthms-1}, $v_1,\ldots,v_4$ are potential vertices of $\mathcal{N}_0(\sigma)$. Notice that
\begin{align*}
v_1&=(3,4,4),\\
v_2&=(4,3,4),\\
v_3&=(2,2,1),\\
v_4&=(3,3,3)=\frac{1}{3}(v_1+v_2+v_3).
\end{align*}
Hence, $v_4=h_1+h_2+h_3$ is not a vertex of $\mathcal{N}_0(\sigma)$. By \eqref{eq:exthms-1}, we conclude that $\mathcal{N}_0(\sigma)=\mathcal{N}_p(\sigma)$. 

Let us now compute a cone of maximal dimension of the normal fan of $\mathcal{N}_0(\sigma)$. 
Let $\sigma_1\subset \sigma \subset N_\mathbb{R}$ be the full-dimensional cone corresponding to the vertex $v_3$. 
Then its dual cone $\sigma_1^\vee\subset M_\mathbb{R}$ is the cone generated by 
$$v_2-v_3=(2,1,3),\quad v_1-v_3=(1,2,3),\quad h_1=(1,0,0),\quad h_2=(0,1,0),\quad\mbox{and}\quad h_3=(2,2,3).$$
Now, $(v_2-v_3)+h_2=h_3$ so $\sigma_1^\vee$ is the cone generated by the columns of the matrix
$$
B_1=\left[\begin{array}{rrrr}
1 & 0 & 1 & 2 \\
0 & 1 & 2 & 1 \\
0 & 0 & 3 & 3 
\end{array}\right]\,.
$$
In particular, this cone is not regular. 
Hence, the normalized Nash blowup of $X(\sigma)$ is singular. 
The same goes for the Nash blowup without normalizing.

Finally, we show that the cone $\sigma_1^\vee$ is not $G$-stable. Indeed, the minimal generating set of $\sigma_1^\vee$ is $G(\sigma_1^\vee)=\{h_1,h_2,h_3,h_4,h_5,h_6\}$ where the first four elements are the ray generators of the cone $\sigma_1^\vee$, $h_5=(1,1,1)$, and $h_6=(1,1,2)$. Now, \cref{def:DCMGS}~$(i)$ does not hold since 
$h_5\notin\Gamma(\sigma_1^\vee)\cap M$. Indeed, the element $h_5$ satisfies 
$$h_5=u_1+u_2,\quad \mbox{where}\quad u_1=\frac13(h_1+h_2+h_6)\in \Gamma(\sigma_1^\vee)\mbox{ and } u_2=\frac13(1,1,1) \in \sigma_1^\vee\setminus\{0\}.$$ 
\end{example}

\begin{remark}
We have seen in the previous example that the singularities of $X(\sigma)$ are not resolved by a single application of the normalized Nash blowup. This contrasts with the two-dimensional case, where the singularities of any affine normal toric surface $X(\tau)$, such that the minimal generating set of $\tau^\vee \subset M_{\mathbb{R}} \simeq \mathbb{R}^2$ is contained in a line (a hyperplane), are resolved by a single application of the Nash blowup. In \cref{section:G flat} we will generalize to arbitrary dimensions this one-step resolution property of such normal toric surfaces by adding certain conditions on the Hilbert basis of a cone (see \cref{thm:altura1}).
\end{remark}

\subsection{Known results on normalized Nash blowups of toric varieties}

In this section we discuss known results on Nash blowups that are particular cases of \cref{th:Nash polyh} and  \cref{cor:Nash-free}. 

Let us start with the case of normal toric surfaces.

\begin{proposition}\label{prop:G-stable-dim-two}
Let $\sigma\subset M_{\R}$ be a two-dimensional cone, where $M\cong\Z^2$. Then $\sigma$ is $G$-stable.    
\end{proposition}
\begin{proof}
Firstly, it is well-known that \cref{def:DCMGS} $(i)$ holds \cite[Proposition 1.21]{oda1983convex}.   

Let $G(\sigma)=\{\gamma_1,\ldots,\gamma_n\}$, ordered counterclockwise. 
Let $A=\{\gamma_{i_1},\ldots,\gamma_{i_r}\}\subset G(\sigma)$, where $1\leq i_1<\cdots<i_r\leq n$. 
Then the extremal rays of $\sigma_A$ are generated by $\gamma_{i_1}$ and $\gamma_{i_r}$, respectively. 
We want to show that 
$$G(\sigma_A)=\{\gamma_{i_1},\gamma_{i_1+1},\cdots,\gamma_{i_2},\gamma_{i_2+1},\ldots,\gamma_{i_r}\}=:\Tilde{A}.$$
If $|A|=1$ the equality holds. 
Suppose $|A|\geq2$. 
It is known that any two consecutive elements of $G(\sigma)$ have determinant one \cite[Proposition 10.2.2]{cox2011toric}. 
This implies that $\Tilde{A}$ minimally generates $\sigma_A\cap M$. 
Hence, \cref{def:DCMGS}~$(ii)$ holds.
\end{proof}

Using the previous proposition, we recover \cite[Theorem 2.5]{DJNB}.

\begin{corollary}
The iteration of normalized Nash blowups resolves the singularities of normal toric surfaces over fields of positive characteristic.
\end{corollary}
\begin{proof}
Let $\sigma\subset M_{\R}$ be a two-dimensional cone defining a normal toric surface $X(\sigma)$. Then, by \cref{prop:G-stable-dim-two}, $\sigma$ is $G$-stable. By \cref{th:Nash polyh}, $\mathcal{N}_0(\sigma)=\mathcal{N}_p(\sigma)$ for all $p>0$ prime. The iteration of normalized Nash blowup of a normal toric surface gives a resolution of singularities in characteristic zero \cite[Section 2.3, Th\'eor\`eme]{GS1}. Hence the result follows by  \cref{cor:Nash-free}.
\end{proof}

Now we study 2-generic determinantal varieties. Let us recall their definition. Let $m,n\in \Z_{\geq2}$. Given a generic $ m \times n $ matrix $L$, let $J_2$ be the ideal generated by the $(2\times2)$-minors of $L$. We denote $M^{2}_{m,n}\subset\K^{mn}$ the zero locus of $J_2$. For general properties of generic determinantal varieties we refer to \cite{EagonHochster,BrunsVetter,Sturm}.

The ideal $J_2$ can be described as a toric ideal as follows. Let $ M $ be a lattice with basis $ e_{1}, \dots, e_{m} $ and $ M '$ another lattice with basis $ f_{1}, \dots, f_{n}$. Let $ \Delta_{m-1} = \operatorname{Conv}\left\{ e_{1}, \dots, e_{m} \right\} \subset M_{\R}$ and $ \Delta_{n-1} = \operatorname{Conv} \left\{ f_{1}, \dots, f_{n} \right\} \subset M'_{\R}$. These polytopes are called the \emph{standard} simplices. They have dimensions $ m-1 $ and $ n-1 $, respectively. 

The Cartesian product $ \Delta_{m-1} \times \Delta_{n-1} $ coincides with $ \operatorname{Conv} \left\{ (e_{i},f_{j}) : 1\leq i\leq m,1\leq j\leq n \right\}\subset M_{\R} \oplus M'_{\R} $. Hence, it is also a lattice polytope. Let $\mathcal{A}=\{(e_{i},f_{j}) : 1\leq i\leq m,1\leq j\leq n\}\subset M \oplus M'$ be the set of vertices of $ \Delta_{m-1} \times \Delta_{n-1} $ and $I_{\mathcal{A}}$ the corresponding toric ideal. Then $J_2=I_{\mathcal{A}}$ \cite[Proposition 5.4]{Sturm}. Moreover, it is also known that $\N\A$ is a saturated semigroup and $\A$ is the Hilbert basis of $\sigma^\vee=\con{\A}$ \cite[Corollaries 1.8 and 1.9]{DDR}. It is known that $M^2_{m,n}$ has dimension $m+n-1$. Hence, the lattice $\mathbb{Z} \mathcal{A}$ has rank $m+n-1$. In particular, it is not equal to $M\oplus M'$.

\begin{proposition}\label{prop:product_simplices}
With the previous notation, we have that $\sigma^\vee$ is $G$-stable. 
\end{proposition}
\begin{proof}

We can assume that $e_i$'s and $f_j$'s are canonical basis elements of $M$ and $M'$, respectively. Letting $y_1,\ldots,y_m$ and $z_1,\ldots,z_n$ be coordinates in $M_{\R}\cong\R^m$ and $M'_{\R}\cong\R^n$, we have that $\A$ is contained in the affine hyperplane $\big\{\sum_{i=1}^my_i+\sum_{j=1}^nz_j=2\big\}\subset M_{\R}\oplus M'_{\R}$. Since $G(\sigma^\vee)=\A$ we obtain that \cref{def:DCMGS} $(i)$ holds.

To show that $\sd$ satisfies \cref{def:DCMGS} $(ii)$, we use the following special property of the product $ \Delta_{m-1} \times \Delta_{n-1} $: the volume of every full-dimensional simplex that can be formed with the elements of $\A$ is the same \cite[Proposition 6.2.11]{de2010triangulations}. Equivalently, letting $D=m+n-1$, every subset of $D+1$ affinely independent elements of $\A$ generates the lattice $\mathbb{Z}\mathcal{A}\cong\Z^{D}$ \cite[Section 2.1]{nill2024unimodularpolytopesnewhellertype}. 

Let $\mathcal{B}\subset\A$. Let $e$ be an element of $\con{\mathcal{B}}\cap\Z^D$. By Caratheodory's theorem in its conic version \cite[Theorem 2.3]{barvinok2002course}, $e$ is a linear combination of a subset $\mathcal{B}_0\subset\mathcal{B}$ of at most $D+1$ elements. We can assume that $e$ belongs to the interior of $\con{\mathcal{B}_0}$. Now extend $\mathcal{B}_0$ to a subset $\mathcal{B}_1\subset\A$ of $D+1$ affinely independent elements as follows. 

Take any $v\in\mathcal{B}_0$. Notice that $\{\mathcal{B}_0-v\}\setminus\{0\}$ is a base of the $\R$-vector space $W=\operatorname{span}_{\R}(\mathcal{B}_0-v)$. Since $W$ is a subspace of $\operatorname{span}_{\R}(\A-v)$, we can complete $\{\mathcal{B}_0-v\}\setminus\{0\}$ to a base of $\beta$ of $\operatorname{span}_{\R}(\A-v)$ using elements of $\{\A-v\}$.
Let $\mathcal{B}_1=\mathcal{B}_0\cup\{\beta+v\}$.

By the property of $ \Delta_{m-1} \times \Delta_{n-1} $ mentioned earlier, $\con{\mathcal{B}_1}$ is regular. In particular, $e$ is an integral combination of $\mathcal{B}_0$. This implies that $\mathcal{B}$ is a generating set of $\con{\mathcal{B}}\cap\Z^D$. Moreover, it is minimal. Thus, \cref{def:DCMGS}~$(ii)$ is satisfied.
\end{proof}

It was proved in \cite[Theorem 3.3]{DDR} that the Nash blowup of $M^2_{m,n}$ (without normalizing) is non-singular in prime characteristic. Using \cref{prop:product_simplices}, we obtain the same result for the normalized Nash blowup.

\begin{corollary}
The normalized Nash blowup of $M^2_{m,n}$ is non-singular over fields of positive characteristic.
\end{corollary}
\begin{proof}
Let $\sigma_1^\vee\subset M_{\R}$ be the cone defining $M^2_{m,n}$. By \cref{prop:product_simplices}, $\sigma_1^\vee$ is $G$-stable. Then, by \cref{th:Nash polyh}, $\mathcal{N}_0(\sigma)=\mathcal{N}_p(\sigma)$ for all $p>0$ prime. The Nash blowup of $M^2_{m,n}$ is non-singular in characteristic zero \cite[Section 1]{EGZ}. In particular, this is also true for its normalization. Hence the result follows by \cref{cor:Nash-free}.
\end{proof}

\section{G-flat Hilbert bases and one-step resolution}\label{section:G flat}

In this section we construct a family of cones which are not necessarily $G$-stable but for which we also have the independence of the characteristic of its normalized Nash blowup. 
Notice that several of the examples in previous sections have their Hilbert basis contained in a hyperplane (see \cref{ex:surf-one-segment}, \cref{ex:polygon}, and \cref{ex:thms-1}). We will further exploit this condition. Our analysis requires the study of some properties of lattice polytopes, that is, a convex polytope whose vertices all have integer coordinates. Let us start with some basic definitions on this context.

\begin{definition}\label{def:unimodular}
    
Two lattice polytopes are \emph{unimodularly equivalent} if one can be transformed into the other by an affine map whose linear part is unimodular, i.e., a map of the form $T(x)=Ax+b$, where $A$ is an integer matrix with determinant $\pm 1$, and the entries of $b$ are integers. 
If $P$ and $Q$ are unimodularly equivalent, then their respective normal fans are isomorphic as fans.

A $d$-simplex is \emph{unimodular} if it is unimodularly equivalent to the standard $d$-simplex $S_d$ given by $\operatorname{Conv}(0,e_1,\ldots,e_d)\subset\R^d$. A \emph{unimodular triangulation} of $P$ is a triangulation of $P$ into simplices such that the union of these simplices exactly covers $P$, their intersections are either empty or common faces, and each simplex is unimodular. 
\end{definition}

Given a lattice polytope $P\subset M_{\R}$, we embed $P$ in $M_{\R}\times\R$ by adding 1 as an extra coordinate and we define the cone
\begin{equation}\label{eq:omega}
     \omega_{P} := \operatorname{Cone} \left\{ (p,1) \mid p \in P \right\} \subset M_{\mathbb{R}}\times \mathbb{R}.
\end{equation}
Remark that, if $P$ and $Q$ are unimodularly equivalent lattice polytopes, then $\omega_P$ is isomorphic to $\omega_Q$ as cones. Indeed, if the affine map $T(x)=Ax+b$ maps $P$ to $Q$ then the map 
$$M_\mathbb{R}\times \mathbb{R}\to M_\mathbb{R}\times \mathbb{R} \quad\mbox{given by}\quad (x,0)\mapsto (Ax,0) \mbox{ and } (0,1)\mapsto (b,1)$$ 
is an automorphism of $M_{\R}\times\R$ mapping $\omega_P$ to $\omega_Q$. 

\begin{definition}\label{def:G-flat}
Let $P \subset M_{\R}$ be a lattice polytope. We say that $P$ is \emph{G-flat} if the Hilbert basis of $ \omega_{P}$, $G(\omega_{P})$, is contained in the hyperplane $ \left\{ (x,1)\in M_{\R}\times\R \mid x\in M_{\mathbb{R}} \right\} $. If $P$ is $G$-flat, then the Hilbert basis of $\omega_P$ is $\{ (m,1) \mid m \in P\cap M\}$.
\end{definition}

There is one more definition we need.

\begin{definition}\label{def:smooth}
Let $P \subset M_{\R}$ be a lattice polytope. 
We say that $P$ is \emph{smooth} if its normal fan is regular, i.e.,  if for every vertex $v\in P$, the cone $\sigma_v^\vee := \operatorname{Cone}\{p-v \mid p\in P\}$ is regular. 
The cone $\sigma_v^\vee$ is called the \emph{feasible cone} of $P$ at $v$ and it is denoted $\operatorname{fcone}(P,v)$.
\end{definition}

\begin{remark}\label{rem:feasible_rem}
It is known that the minimal generators for the feasible cone $\operatorname{fcone}(P,v)$ are the edges (both bounded and unbounded) adjacent to $v$.
\end{remark}

\begin{example}[Relations between definitions]\label{ex:relations}
The following are some relations between the named properties of lattice polytopes.
\begin{description}
    \item[Unimodular triangulation $\Longrightarrow$ $G$-flat] Since $P$ admits a unimodular triangulation, it follows that $\omega_P$  admits a regular subdivision, $\Sigma$, such that $\Sigma(1) = \{ (m, 1) \mid m \in P \cap M \}$. From this, we obtain that $G(\omega_P) = \{ (m, 1) \mid m \in P \cap M \}$.
    \item[Unimodular triangulation $\centernot\Longrightarrow$ $G$-stable] 
Consider the three dimensional cube $ [0,1]^{3} $. It has a unimodular triangulation.	Indeed, the six tetrahedra defined by
	\begin{equation}
		T_{\pi} = \operatorname{Conv} \left\{ ( x_{1}, x_{2}, x_{3} ) \in \mathbb{R}^{3} ~:~ 0\leq x_{\pi(1)} \leq x_{\pi(2)} \leq x_{\pi(3)} \leq 1 \right\},
	\end{equation}
for each permutation $ \pi $ of $ \left\{ 1,2,3 \right\}$, form a unimodular triangulation. However it is not $G$-stable as it fails \cref{def:DCMGS}~$(ii)$: the subcone of $ \omega_{P} $ spanned by the vectors \linebreak $(1,0,0,1), (0,1,0,1), (0,0,1,1), (1,1,1,1) $ has the element $(1,1,1,2)$ in its Hilbert basis, which is not in the Hilbert basis of $ \omega_{P} $ because its last coordinate is 2.

    \item[Smooth $\centernot \Longrightarrow$ $G$-stable] Same as above since the cube is smooth.
    \item[$G$-flat $\centernot \Longrightarrow$ smooth] Every lattice polygon has unimodular triangulation, hence is $G$-flat. This happens even if they are not smooth.
\end{description}
On the other hand, it is not known whether smooth implies $G$-flat.
This is an open question in dimensions three and higher \cite[Section 1.5.1]{haase2021existence}\footnote{What we call G-flat is property (7) in the hierarchy of definitions in \cite[Section 1.2.5]{haase2021existence}}.
It is also open whether smooth polytopes always have a unimodular triangulation.
\end{example}

In order to avoid introducing excessive notation, we will use $\npp$ on this section (compare with~\eqref{eq:N_polyhedron}). Our next goal is to prove that the smoothness and $G$-flatness of $P$ is a sufficient condition for having $\npc=\npp$ for all $p>0$ (see \cref{cor:char_free}). Moreover, this result will allow us to show that the normalized Nash blowup of the toric variety defined by $\omega_{P}$ is non-singular over arbitrary characteristic fields under some conditions on $P$ (see \cref{thm:altura1}). These results will require an analysis of barycentric hulls of lattice polytopes. This is the content of the following section.

\subsection{Barycentric hull}

In this section we define and derive some results about the barycentric hull of a lattice polytope. These results are important for \cref{ssec:snipe} although they may be of independent interest.

\begin{definition} \label{def:barycentric-hull}
Let $\{ v_1, \dots, v_m \} $ be the vertices of a $d$-polytope $P$.
The \emph{barycenter} of $P$ is the point $\frac{1}{m} \sum_i v_i$.
The \emph{barycentric hull} $ \mathsf{B}(P) $ of a lattice $ d $-polytope $P$ is the convex hull of the barycenters of all possible $ d $-simplices with vertices in $ P \cap M $ .
\end{definition}

\begin{figure}
\begin{minipage}[c]{0.49\linewidth}
	\centering
\includegraphics[width=0.6\linewidth]{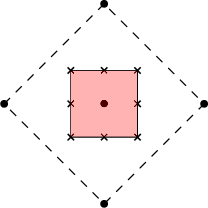}
\caption{Barycentric hull of the rhombus (the points in $P\cap M$ are dots and the barycenters are crosses, the origin is not a barycenter).}
\label{fig:rombo}
\end{minipage}
\begin{minipage}[c]{0.49\linewidth}
	\centering
\includegraphics[width=0.6\linewidth]{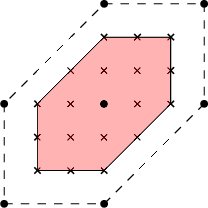}
\caption{Barycentric hull of a smooth polygon (the points in $P\cap M$ are dots and the barycenters are crosses, the origin is a barycenter).}
\label{fig:smooth}
\end{minipage}
\end{figure}

\begin{example}\label{ex:new rhombus}
   Let $P$ be the rhombus as in \cref{fig:rombo}. It is the convex hull of the points given by the columns of the matrix
   \[
   \begin{bmatrix}
       1&-1&0&0 \\ 0&0&1&-1
   \end{bmatrix}.
   \]
    We depict its barycentric hull in \cref{fig:rombo}.
    Notice that the convex hull of the barycenters is a smooth square, whereas the original rhombus is not smooth. Notice that $P$ is $G$-flat.
\end{example}

\begin{example}
    Now consider a smooth polygon as in \cref{fig:smooth}.
    It is the convex hull of the points given by the columns of the matrix
   \[
   \begin{bmatrix}
       0&1&1&0&-1&-1&0 \\ 0&0&1&1&0&-1&-1
   \end{bmatrix}.
   \]
    In this case there are seven lattice points which leads to 35 triples, but three of those triples are collinear points.
    So there are 32 non-degenerate triangles, although some have the same barycenter.
    The origin is both a lattice point of $P$ and a barycenter of two triangles.

\end{example}

We define a \textit{corner} of a smooth $d$-polytope $P$ to be the unimodular simplex obtained as the convex hull of the $(d+1)$-tuple $ (v_{0},\dots, v_{d}) $ of lattice points in $ P\cap M$, where $ v_{0} $ is a vertex of $P$, and $ v_{1}, \dots, v_{d} $ are the nearest lattice points to $ v_{0} $ laying on the $d$ edges adjacent to $v_0$ (see \cref{fig:corners}).
If $ P $ is a unimodular simplex, then $ P $ is itself its only corner. 
If $P$ is not a unimodular simplex, then there is a different corner for each vertex of $P$.
We denote by $ \mathcal{C}(v) $ the barycenter of the corner associated to the vertex $ v \in P $.

\begin{figure}
    \centering
    \includegraphics[width=0.2\linewidth]{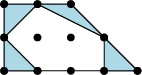}
    \caption{A polygon with its corners highlighted.}
    \label{fig:corners}
\end{figure}

\begin{theorem}\label{thm:baryhull}
Let $P$ be a smooth lattice polytope. Then its barycentric hull $\mathsf{B}(P)$  is the convex hull of the barycenters of the corners of $P$.

Furthermore, if $P$ is a unimodular simplex, then $\mathsf{B}(P)$ is a point. 
Else, if $P$ is not a unimodular simplex, then the function $ \mathcal{C}: \operatorname{Verts}(P) \to \operatorname{Verts}( \mathsf{B}(P)) $, $v\mapsto \mathcal{C}(v)$, is a bijection,
and the edge between $ v $ and $ w $ is parallel to the edge between $ \mathcal{C}(v) $ and $ \mathcal{C}(w) $.
As a consequence, $P$ and $\mathsf{B}(P)$ have the same collection of feasible cones.
More precisely,
\begin{equation}\label{eq:fcones}
    \operatorname{fcone}(P,v) = \operatorname{fcone}(\mathsf{B}(P), \mathcal{C}(v)).
\end{equation}
In particular, $\mathsf{B}(P)$ is a smooth polytope.
\end{theorem}

Before proving the theorem, let us state the following consequence.

\begin{corollary}\label{cor:char_free}
Let $P \subset M_{\R}$ be a smooth and $G$-flat $d$-lattice polytope. 
We have $\npc=\npp$ for every prime $ p $.
\end{corollary}

\begin{proof}
	Since $P$ is $G$-flat, the Hilbert basis of $\omega_P$ is $\{(v,1)|v\in P\cap M\}$. Then, it follows from the definitions that the vertices of $ \npc $ are the vertices of the polytope $ (d+1)(\mathsf{B}(P),1) \subseteq M_{\R}\times \R $.
	By \cref{thm:baryhull}, these vertices are the sum of the corners of $ (P,1) $.
	Since $ P $ is smooth, for each corner $ \left\{ v_{0}, \dots, v_{d} \right\} $ the set $ \left\{ (v_{0},1) , \dots, (v_{d},1) \right\} $ is a basis for $M\times\Z$.
	Hence the associated determinants are $ \pm 1 $.
	The conclusion follows.
\end{proof}

For the proof of the theorem we need some basic auxiliary lemmas about smooth polytopes. 
Remark that, up to unimodular equivalence, all $d$-dimensional unimodular simplices are equivalent to the standard simplex 
$ S_d \subseteq \mathbb{R}^d$.

\begin{lemma}\label{lem:aux_smooth_simplices}
If $P$ is a smooth $d$-simplex, then, up to unimodular equivalence, $P$ is equivalent to $k\cdot S_d$ for some positive integer $k$.
\end{lemma}

\begin{proof}
Let $v \in P$ be a vertex. Applying a unimodular transformation we can assume that $v=0$ and all the other vertices are $k_1e_1, \dots, k_de_d$ for some positive integers $ k_1, \dots, k_d$. The feasible cone at the vertex $k_1e_1$ is generated by  $-k_1e_1$ and $k_ie_i-k_1e_1$. Such cone is regular if and only if $k_i=k_1=:k$ for all $i\geq 2$.
\end{proof}

\begin{lemma}\label{lem:aux_pyramid}
Let $P$ be a smooth lattice polytope which is not the standard simplex.
Furthermore, assume that the standard simplex $S_d$ is a corner of $P$, for some $d\geq2$.
Then for every $i$ there exists a $j$ such that $e_i + e_j \in P$.
\end{lemma}

\begin{proof}
We argue by induction on $d$.
Let $d=2$. 
Since $S_2$ is a corner, there are vertices of the form $me_1$ and $ne_2$ for some positive integers $m,n$.
If $e_1+e_2 \notin P$, then by convexity, $P$ contains no elements of the form $a_1e_1+a_2e_2$ with integers $a_1,a_2>0$. 
It follows that $P$ has only three vertices: the origin, and one on each axis, hence it is a triangle.
By \cref{lem:aux_smooth_simplices}, it is a multiple of $S_2$, but the only multiple which does not contain $e_1+e_2$ is $S_2$ itself proving the lemma for $d=2$.

Now assume the lemma holds for smooth polytopes of dimension less than $d$.
Since $S_d$ is a corner, there are vertices of the form $n_ke_k$ for every $k$. 
Assume that there exists $i$ such that  $e_i + e_j\notin P$ for all $j\neq i$. 

For every $j \neq i$ consider the two-dimensional face $G_j : = P \cap \{ x \in \mathbb{R}^d \mid x_k = 0, \forall k \neq i,j \} $.
The polygon $G_j$ is a smooth polytope of dimension $2<d$ having $\operatorname{Conv}(0,e_i,e_j)=S_2$ as a corner. Moreover, $e_i+e_j\notin G_j$. By the inductive hypothesis we conclude that $G_j$ is the standard simplex $S_2$. It follows that the vertices of $G_j$ are $\{0, e_i, e_j \}$. Consequently, $P$ has vertices $\{0,e_1, \dots, e_d\}$ and the neighbors of $e_i$ are exactly $\{0,e_1, \dots, e_d\} \setminus \{e_i\}$; there cannot be more neighbors since a smooth polytope has to be simple.

By hypothesis, $P$ is contained in the positive orthant. We conclude that it is a pyramid with apex $e_i$ over the facet $F = P \cap \{ x \in \mathbb{R}^d \mid x_i = 0 \} $.
For a pyramid to be a simple polytope we must have that the base $F$ is simplex, and so $P$ itself is a simplex.
It follows from \cref{lem:aux_smooth_simplices} that it is a multiple of $S_d$.
Since it contains $e_i$ the scalar multiple must be equal to one, proving the lemma.
\end{proof}

We are now ready to prove the main result of this section. 

\begin{proof}[Proof of \cref{thm:baryhull}]
Let $ d = \dim(P) $.
If $d=1$, $P$ is just a closed segment in $\R$, its corners are the segments joining the vertices of $P$ with their nearest integers inside $P$ and the theorem follows. 
Hence in the sequel we assume $d>1$. 
 If $P$ is unimodularly equivalent to $S_{d}$ then there is a single corner and the theorem follows. 
We assume in the sequel that $P$ is not unimodular equivalent to $S_{d}$.

We argue one vertex at the time. 
We may and will assume that $ v=0 $ and, since $P$ is smooth, that its corner is the standard simplex $S_d$. We first show that the barycenter of this corner, namely $ \mathbf{p} = \frac1{d+1} ( e_{1} + \dots + e_{d}) $, is a vertex of $\mathsf{B}(P)$.
Notice that $P$ lies in the positive orthant.

Let $i\in \{1,\dots,d\}$. Let $ j\neq i$ be such that $ e_{i}+e_{j} \in P $, which exists by \cref{lem:aux_pyramid}. 
Hence, replacing $e_j$ by $e_i+e_j$ we have that 
\begin{align*}
  \mathbf{q}_i = \frac1{d+1} (e_{1} + \dots + e_{i-1}+2e_i+e_{i+1}+\dots + e_{d} )
\end{align*}
is the barycenter of another $d$-simplex with vertices in $P\cap M$. 
Since $P$ lies in the positive orthant, the barycenter of every $d$-simplex with vertices in $P\cap M$ has all its coordinates greater or equal than $ \frac{1}{d+1}$.
The entries of $\mathbf{p}$ and $\mathbf{q_i}$ are the same and equal to $\frac{1}{d+1}$ except at the $i$th entry. 
We conclude that $ \mathbf{p}$ and $\mathbf{q}_i$ lie on an edge $E_i$ of the barycentric hull of $P$.
The direction of $E_i$ is given by $e_{i}$ and $\mathbf{p}$ is one of its vertices. 
In particular, $\mathbf{p}$ is a vertex of $\mathsf{B}(P)$.

Let now $ a $ be the largest integer such that $ ae_{i} \in P $. 
The point $ ae_{i} $ is a vertex of $P$. 
We claim that the other vertex  of the edge $E_i$ is the barycenter of the corner of $ae_{i}$. 

Indeed, let $\mathscr{T}$ be the set of all $d$-simplices with vertices in $P\cap M$ such that its barycenter has coordinates in $ e_{1},\dots, e_{i-1},e_{i+1}\dots, e_{d} $ equal to $\frac1{d+1}$. This set is nonempty as it contains $S_d$, so it has an element $T$ with maximal $e_i$ coordinate. Clearly, the other vertex  of the edge $E_i$ is the barycenter of $T$. We prove that $T$ is the corner of $ ae_{i} $.

Fix $ j\neq i$.	
Since the $j$-coordinate of the barycenter of $T$ is $\frac1{d+1}$, we have that $ T $ has $ d $ of its vertices in the coordinate hyperplane $ H_{j} := \left\{ x \in \mathbb{R}^{d} \mid x_{j} = 0 \right\} $ and the last vertex is in the hyperplane $ \left\{ x \in \mathbb{R}^{d} \mid x_{j} = 1 \right\} $. 
Since this is true for every $ j\neq i$, we have that $ d-1  $ vertices of  $ T $ are of the form $ c_{j}e_{i} + e_{j} $ for some nonnegative integer $ c_{j} $. 
The other two vertices must be of the form $ c_0e_{i} $ and $ c_ie_{i} $. 
By the maximality of $ T $, we have that $\sum_{j=0}^d c_j $ is as large as possible so each $c_j$ is as large as possible. 

For $ c_{0} $ and $ c_{i} $ the maximal choices are $ a-1 $ and $ a $. In particular, $(a-1)e_i$ is the nearest lattice point to $ae_i$ laying on this edge of $P$.

For each $ j\neq 0,i,$ consider the two-dimensional face $G_j = P\cap\left\{ x \in \mathbb{R}^{d} \mid x_{k} = 0, \forall k \neq i,j \right\} $.
It is two-dimensional since it contains the points $\{ 0, e_i, e_j \}$.
The polygon $G_j$ has $ae_i$ as a vertex.
This vertex has two adjacent edges: one in direction $e_i$ and another, say $\tilde{E}$.
The lattice point in $G_j$ of the form $c_{j}e_{i} + e_{j}$ and largest possible $c_j$, is the nearest lattice point $ae_i$ in the edge $\tilde{E}$.

We conclude that the vertices of $ T $ are given by $ ae_{1} $ and all of its nearest neighbors, i.e., $T$ is the corner of $ae_{1}$ as claimed.

In conclusion, we have $d$ edges adjacent to the vertex $\mathbf{p}$ of $\mathsf{B}(P)$.
We claim there are no more edges adjacent to $\mathbf{p}$.
With the edges we already verified, we know that the feasible cone at that vertex $\mathbf{p}$ contains $\mathbb{R}^d_{\geq 0}$.
Consider an arbitrary barycenter $\mathbf{b}$ of a $d$-simplex with vertices in $P\cap M$.
It has every coordinate greater or equal than $\frac{1}{d+1}$.
This means that the vector $\mathbf{b}-\mathbf{p}$ is contained in $\mathbb{R}^d_{\geq 0}$.
Hence by Remark \ref{rem:feasible_rem}, the feasible cone at $\mathbf{p}$ is exactly the positive orthant, and we have found all the adjacent edges.

We have proven that for every vertex $v$ of $P$, the barycenter of its corner, $\mathcal{C}(v)$, is a vertex of $\mathsf{B}(P)$. 
We also showed that for each vertex $v$ of $P$, there are exactly $d$ edges $E_1,\ldots,E_d$ of $\mathsf{B}(P)$ adjacent to $\mathcal{C}(v)$ with the following property: the two vertices of $E_i$ are $\mathcal{C}(v)$ and $\mathcal{C}(w_i)$, for some corner $w_i$ of $P$. 
Consider the graph formed by all the vertices and edges of $\mathsf{B}(P)$. 
Since it is a connected graph, we conclude that every vertex of $\mathsf{B}(P)$ is of the form $\mathcal{C}(v)$, for a corner $v$ of $P$. 
Therefore, the vertices of the barycentric hull of $P$ are given by the barycenters of its corners. 

In particular, the map $ \mathcal{C}: \operatorname{Verts}(P) \to \operatorname{Verts}( \mathsf{B}(P)) $ is a bijection and edge directions are parallel to the corresponding edge direction between the corners in the original $ P $. This proves the last part of the theorem.
\end{proof}

\subsection{One-step resolution via normalized Nash blowups in arbitrary characteristic}
\label{ssec:snipe}

In this section we provide a family of cones whose corresponding toric variety is resolved with a single normalized Nash blowup over fields of arbitrary characteristic.

The following example illustrates the strategy for the main theorem of this section.
\begin{example}
Let $P$ be the unit square in $\R^2$, i.e., the convex hull of the points given by the columns of the matrix
   \[
   \begin{bmatrix}
       0&1&0&1 \\ 0&0&1&1
   \end{bmatrix}.
   \]
The barycentric hull of $P$ is the convex hull of the points given by the columns of the matrix
   \[
   \begin{bmatrix}
       1/3&2/3&1/3&2/3 \\ 1/3&1/3&2/3&2/3
   \end{bmatrix}.
   \]
The Hilbert basis of $\omega_P$ is the set $\{(0,0,1),(1,0,1),(0,1,1),(1,1,1)\}.$ By definition or by applying \cref{thm:baryhull},
$$\npc=\operatorname{Conv}(\{(1,1,3),(2,1,3),(1,2,3),(2,2,3)\}+\omega_P).$$
By \cref{thm:baryhull}, the vertices of $\npc$ are these same vectors. Let us verify that the feasible cone of $\npc$ at $(1,2,3)$, i.e., $K=\operatorname{Cone}(\npc-(1,2,3))$, is regular. A straightforward computation shows that $K$ can be generated by $(0,-1,0)$, $(1,0,0)$, $(0,1,1)$. Hence, it is regular. 

Notice that $(0,1)$ is the vertex of $P$ that corresponds to $(1,2,3)$ under the correspondence of \cref{thm:baryhull}. The key remark here is that the generators of $K$ are completely determined by the vertex $(0,1)$ and its neighbors in $P$.  

The same computations apply to the other vertices. By \cref{th:Nash-blowup}, the normalized Nash blowup of $X(\omega_P)$ is non-singular over fields of characteristic zero. By \cref{cor:char_free}, this is true also in positive characteristic.
\end{example}

\begin{theorem}\label{thm:altura1}
Let $ P $ be a smooth and G-flat lattice polytope. Let $X$ be the normal toric variety defined by the semigroup $\mathcal{S}(\omega_{P})$. Then the normalized Nash blowup of $X$ is non-singular over fields of arbitrary characteristic.
\end{theorem}

\begin{proof}
By \cref{cor:char_free}, it is enough to prove the theorem over fields of characteristic zero. We prove that the feasible cone at every vertex of $\npc$ is regular. The conclusion then follows from \cref{th:Nash-blowup}.

Recall the description of $\npc$ (see (\ref{eq:N_polyhedron}), \cref{thm:baryhull}, and the proof of \cref{cor:char_free}):
the vertices of $\npc$ coincide with the vertices of the polytope $ \mathsf{C}(P) := (d+1)(\mathsf{B}(P),1)$;
in turn, the vertices of $\mathsf{C}(P)$ correspond to vertices defining corners of $P$ and its edges are parallel to those of $P$;
finally, starting at each vertex of $\npc$, there are rays along the directions of $(v,1)$ for some vertices $v$ of $P$ (which rays depend on the vertex of $\npc$). 

Let $\mathsf{C}(v) := (d+1)(\mathcal{C}(v),1)$ be a vertex of $\npc$, for some vertex $v$ defining a corner of $P$. 
Let $K$ be the feasible cone of $\npc$ at $\mathsf{C}(v)$, i.e., it is the cone generated by $\npc-\mathsf{C}(v)$. 
We claim that $K$ is equal to $\operatorname{fcone}(P,v) + \operatorname{Cone}\{(v,1)\}$.
By the coincidence of vertices of $\npc$ and $\mathsf{C}(v)$ we have
$$
K = \operatorname{fcone}(\npc,\mathsf{C}(v))=\operatorname{fcone}(\mathsf{C}(P)+\omega_P,\mathsf{C}(v))=\operatorname{fcone}(\mathsf{C}(P), \mathsf{C}(v)) + \omega_P.
$$
By Equation \eqref{eq:fcones}, we have 
$\operatorname{fcone}(\mathsf{C}(P), \mathsf{C}(v)) = \operatorname{fcone}((P,1),(v,1))$.
It follows that $K$ is equal to $\operatorname{fcone}((P,1),(v,1))+\omega_P=\operatorname{fcone}((P,1)+\omega_P,(v,1))$.
Notice that the polyhedron $(P,1)+\omega_P$ can be described as a truncation of $\omega_P$. Indeed, $(P,1)+\omega_P = \omega_P \cap (M_\mathbb{R}\times\mathbb{R}_{\geq1})$. Since the rays of $\omega_P$ are exactly the ones generated by $\{(p,1)\mid p\mbox{ a vertex of } P\}$, we conclude that 
$$K=\operatorname{fcone}(\omega_P \cap (M_\mathbb{R}\times\mathbb{R}_{\geq1}),(v,1))=\operatorname{fcone}(P,v) + \operatorname{Cone}\{(v,1)\}.$$
This finishes the proof of the claim. Now consider the following set
$$A=\{(u,1)-(v,1)\mid u\text{ is a vertex of the corner of  }v \}\cup\{(v,1)\}\,.$$
Since all edges adjacent to $v$ in $P$ are determined by the vertices different from $v$ in the corner of $v$, \cref{rem:feasible_rem} implies that $K$ is generated by $A$.

Finally, the set of vertices of the corner of $v$ different from $v$ forms a basis for $M$ by smoothness of $P$. Together with $(v,1)$, it forms a basis for $M\times \mathbb{Z}$. 
Thus, $K$ is regular. This proves the theorem in the case of characteristic zero. 
\end{proof}

\begin{remark}
The $G$-flat condition has to be in the hypothesis of Theorem \ref{thm:altura1} for the proof, but again, we point out that there are no known examples of smooth but non $G$-flat polytopes.
\end{remark}

We finish the paper with a final example showing a family of non-smooth polytopes that satisfies the conclusion of \cref{thm:altura1}.

\begin{example}
Let $ P $ be the lattice triangle with vertices $(0,0),(n,0),(0,1)$ with $n>1$. The polyhedron $ \npc $ has two vertices: $ (1,1,3) $ and $ (2n-1,1,3) $.
\begin{enumerate}
\item The feasible cone at $ (1,1,3) $ is spanned by the 3 rays from $ \omega_{P} $, $(0,0,1),(n,0,1),(0,1,0)$, and the difference $ (2n-1,1,3)-(1,1,3) = (2n-2,0,0)$. The resulting cone is the first octant which is unimodular.
\item Analogously, the feasible cone at $ (2n-1,1,3) $ is spanned by $(0,0,1), (n,0,1),(0,1,0)$, and the difference $ (1,1,3)-(2n-1,1,3) = (2-2n,0,0)$. Since $ (n,0,1) + n(-1,0,0) = (0,0,1) $, the resulting cone is spanned by the vectors $ (n,0,1),(0,1,0),(-1,0,0) $, which is unimodular for every $n>1$.
\end{enumerate}
Hence the singularities of the 3-dimensional affine toric variety defined by $\sip$ are resolved by a single normalized Nash blowup. This provides evidence towards an affirmative answer to the question of whether normalized Nash blowups resolves singularities in dimension 3.
\end{example}

\bibliographystyle{alpha}
\bibliography{ref}

\end{document}